\documentclass{scrartcl}

\usepackage{graphicx}
\usepackage{pythontex}
\usepackage{url}
\usepackage{mathtools}
\usepackage{hyperref}
\usepackage{amsmath}
\usepackage{amssymb}
\usepackage{bbm}
\usepackage{amsthm}
\usepackage{mathrsfs} 
\usepackage{paralist}
\usepackage{comment}

\DeclareMathOperator{\supp}{supp}
\DeclareMathOperator{\Law}{Law}

\newtheorem{thm}{Theorem}[section]
\newtheorem*{thm*}{Main Theorem}
\newtheorem{defi}[thm]{Definition}
\newtheorem{lem}[thm]{Lemma}
\newtheorem{rmk}[thm]{Remark}
\newtheorem{prop}[thm]{Proposition}

\newcommand{\N}{\mathbb{N}}
\newcommand{\Z}{\mathbb{Z}}

\newcommand{\R}{\mathbb{R}}
\newcommand{\C}{\mathbb{C}}
\newcommand{\T}{\mathbb{T}}
\newcommand{\calF}{\mathcal{F}}

\title{A Variational Characterization of a Gibbs Measure related to the Aviles-Giga Functional}
\author{Jean-Dominique Deuschel, Clara Thierbach, and Barbara Zwicknagl}
\begin{document}

\maketitle
\textbf{Abstract. }We derive explicit characterizations of Gibbs-measures related to the Aviles-Giga functional. We build on the variational approach presented by Barashkov and Gubinelli \cite{Barashkov_2020} in the context of the $\Phi_3^4$-model.
\thispagestyle{empty}

\section{Introduction}
We consider the energy functional
\begin{eqnarray}\label{eq:H}
  H(\varphi)\coloneq \int_{\mathbb{T}^d}\left(|\nabla\varphi|^4-|\nabla\varphi|^2+|\Delta\varphi|^2\right)\,dx
\end{eqnarray}
on the $d$-dimensional torus $\T^d$.
This functional is closely related to the well-known Aviles-Giga functional on the torus $\T^d$, given by 
\begin{eqnarray}\label{eq:aviles-giga}
J_{AG}(\varphi):=\int_{\T^d} \left((1-|\nabla \varphi|^2)^2+\tau^2|D^2\varphi|^2\right)\,dx =\int_{\T^d} \left((1-|\nabla \varphi|^2)^2+\tau^2|\Delta \varphi|^2\right)\,dx.
\end{eqnarray}
More precisely, both are functionals of the form $\int_{\T^d}\left(W(\nabla \varphi)+\tau^2 |\Delta \varphi|\right)\,dx$, where the set of minimizers of $W$ is given by a sphere $|v|=const$.
 Therefore, they serve as 
models for Landau theories of phase transitions with a vectorial order parameter $\nabla \varphi$. The functional \eqref{eq:aviles-giga} was introduced in \cite{aviles-giga:1987} and  has been extensively studied due to its various applications to models in materials sciences. The latter include for example models for thin magnetic films, liquid crystals,
as well as applications to blistering in thin films (see e.g. \cite{jin-kohn:2000,kohn:2007,ortiz-gioia:1994} and the references therein). \\

In this paper, we aim to contribute to the study of the functional \eqref{eq:H} by characterizing explicitly the related Gibbs measure in a variational setting, following the seminal approach developed in \cite{Barashkov_2020} in the context of the $\Phi_d^4$-model. Roughly speaking, the functional $H$ can be related to a higher-order variant of the $\Phi_d^4$-model with the scalar function $v$ replaced by the vectorial quantity $\nabla v$.
\\
Let us briefly outline the approach from \cite{Barashkov_2020} adapted to our setting. We aim to characterize the Gibbs measure that is heuristically written as 
\begin{align*}
    \nu(d\varphi) =\mathscr{Z}^{-1}e^{-H(\varphi)}\mathscr{L}(d\varphi)
\end{align*}
with a normalization constant $\mathscr{Z}$ and the energy functional \eqref{eq:H}. 
 Here, $\mathscr{L}$ would play the role of a Lebesgue measure on the space of distributions, which does not exist on an infinite dimensional space. \\
 To have a closer connection to the $\Phi^4_d$ model we first consider the Hamiltonian 
 \begin{align}\label{eq:firstHamiltonian}
    H(\varphi)\coloneq \int_{\mathbb{T}^d}\left(|\nabla\varphi|^4+|\nabla\varphi|^2+|\Delta\varphi|^2\right)\ dx
\end{align}
 for $d=2$ and comment on the other Hamiltonian \eqref{eq:H} and the case $d=3$ afterwards.\\
 To overcome the difficulty due to the missing Lebesgue measure we rewrite the quadratic part of the Hamiltonian as a Gaussian measure $\theta$ with covariance $(-\Delta+\Delta^2)^{-1}$. This measure is supported in $\mathscr{C}^{1-\kappa}(\mathbb{T}^2)$ for any $\kappa>0$. (For the readers' convenience, we collect some details in Appendix \ref{GaussianMeasure}).\\
 A more severe problem, however, is that non-linear actions are not well-defined on distributions. Quadratic actions could be handled through a Gaussian measure as in the previous paragraph, but $|\nabla\cdot|^4$ is neither linear nor quadratic. We solve this problem with approximating the distributions by cutting high frequencies in the Fourier space, i.e., with a truncation parameter $T>0$, we consider 
 \begin{align*}
     \tilde{\nu}_T(d\varphi)=\mathscr{Z}_T^{-1}\ e^{-\int_{\mathbb{T}^2}|\nabla\varphi_T|^4\ dx} \theta(d\varphi)
 \end{align*}
which is a well defined expression.
 In order to take the limit $T\to\infty$ we need to renormalize the Hamiltonian and subtract the divergent parts of the singular term $|\nabla\varphi|^4$. We will see later that the renormalization is of order $|\nabla\varphi|^2$, which brings us to the measures
 \begin{align*}
     \nu_T(d\varphi)=\frac{e^{-\int_{\mathbb{T}^2}|\nabla\varphi_T|^4-\mathcal{O}_T(|\nabla\varphi_T|^2)\ dx}}{\mathscr{Z}_T}\theta(d\varphi).
 \end{align*}
Here the constants in $\mathcal{O}_T(|\nabla\varphi_T|^2)$ differ for every $T\in[0,\infty)$. \\
Since we want to determine $\nu$ as the weak limit of $(\nu_T)_T$, the aim of this paper is to show the convergence of these measures. We also get an explicit description for the limit. \\\\
In this regard we define for $f\in C(\mathscr{C}^{1-\kappa}({\mathbb{T}^2});\R)$ with linear growth the free energy 
\begin{align}\label{freeenergy}
    W_T(f)\coloneq - \log\int_{\mathcal{S}'}\exp(-V^f_T(\varphi_T))\ \theta(d\varphi)
\end{align}
with $V^f_T(\varphi)\coloneq f(\varphi)+\int_{\mathbb{T}^2}(|\nabla\varphi(x)|^4- \mathcal{O}_T(|\nabla\varphi(x)|^2))\ dx$. If one shows the convergence of these energies to a limit $W(f)$ we obtain tightness of $(\nu_T)_T$, hence convergence of a subsequence, and an explicit characterization for $\nu$ by its Laplace transform
 \begin{align*}
     \int_{\mathcal{S}'(\Lambda)}e^{f(\varphi)}\ \nu(d\varphi)=\exp(-(W(f)-W(0))).
 \end{align*}
In this way, we reduce the question of the convergence of $(\nu_T)_T$ to the convergence of $(W_T(f))_T$. 
To show this convergence we first derive the variational formulation, heuristically written as
\begin{align}\label{eq:introductionF}
    W_T(f)=\inf_u F_T(u).
\end{align}
Note that the functionals $(F_T)_T$ depend on $f$, and the explicit form of $F_T$ and its domain will be elaborated in the subsequent sections, see \eqref{eq:FT}. \\
To argue for the convergence of these infima (see \eqref{eq:introductionF}), we make use of the fundamental theorem of $\Gamma$-convergence. We consider $\Gamma$-convergence on a topological space which is not necessarily equipped with a metric. For details and proofs we refer to \cite{maso1993introduction}.\\
\begin{defi}
    Let $\mathcal{T}$ be a topological space satisfying the first axiom of countability and let $F,F_n:\mathcal{T}\to[-\infty,\infty]$
    \begin{enumerate}
        \item 
    We say that the sequence of functionals $(F_n)_n$ $\Gamma$-converges to $F$ if
    \begin{itemize}
        \item ($\liminf$-inequality) For every sequence $x_n\to x$ in $\mathcal{T}$
        \begin{align*}
            F(x)\leq\liminf_{n\to\infty}F_n(x_n),
        \end{align*}
        \item($\limsup$-inequality) For every point $x$ there exists a sequence $x_n\to x$ in $\mathcal{T}$(recovery sequence) such that
        \begin{align*}
            F(x)\geq \limsup_{n\to\infty}F_n(x_n).
        \end{align*}
    \end{itemize}

    \item A sequence of functionals $(F_n)_n$ is called equicoercive if there exists a compact set $\mathcal{K}\subset\mathcal{T}$ such that for all $n\in\N$
    \begin{align*}
        \inf_{x\in\mathcal{K}}F_n(x)=\inf_{x\in\mathcal{T}}F_n(x).
    \end{align*}
    \end{enumerate}
\end{defi}
\begin{thm}\label{minimizer}  Let $\mathcal{T}$ be a topological space satisfying the first axiom of countability and let $F,F_n:\mathcal{T}\to[-\infty,\infty]$.
    If $(F_n)_n$ $\Gamma$-converges to $F$ and $(F_n)_n$ is equicoercive, then $F$ attains its minimum and 
    \begin{align*}
        \min_\mathcal{T}F=\lim_{n\to\infty}\inf_\mathcal{T}F_n.
    \end{align*}
\end{thm}
We make use of these concepts and show equicoercivity and $\Gamma$-convergence for the functionals $(F_T)_T$ in \eqref{eq:introductionF}.\\\\
Following this outline the two main steps are to find a variational formulation for the free energies and to show $\Gamma$-convergence for the resulting functionals. Regarding the first point we present an interpretation in the sense of convex analysis using the Fenchel biconjugate. For the $\Gamma$-convergence we need to choose renormalization constants. Since the gradient has a vectorial structure the renormalization differs from the scalar case in the $\Phi^4_2$ model.

\subsection{Notation and setting}
We will use the following notational conventions. 
For any function $h:\R\to\R$, $h(D)$ is Fourier multiplier operator with symbol $h$, i.e., $h(D)f=\calF^{-1}h\calF f$, where $\calF$ denotes the Fourier transform on the torus.\\
If no domain is given for a norm or an integral, we always mean $\T^2$.\\
For any function space $X(\T^d)$, we write $\dot{X}(\T^d)$ for the subspace of functions in $X(\T^d)$ with zero mean.\\
The notation $\sim$ and $\lesssim$, respectively, indicates that an (in)equality holds up to uniform constants that do not depend on the relevant parameters. \\
For any function $f$ on $\mathbb{T}^d$ we write $\partial_if$ for the $i$-th partial derivative and if $f$ is vector-valued $\nabla\cdot f$ for the divergence. For functions $g$ on $X\times [0,\infty)$, where $X$ is an arbitrary space, we write $\dot{g}$ for the partial derivative in $t\in [0,\infty)$.\\
For a vector $v\in\R^d$ and $k\in\N$, we set $v^k=|v|^{k-1}v$ if $k\in\N$ is odd, and $v^k=|v|^k$ if $k\in\N$ is even.\\
For standard notation regarding Besov spaces, we refer to Section \ref{sec:besov}.
\subsection{Gaussian Process}

 We first construct a process $Y$ with values in $\mathscr{C}^{\frac{4-d}{2}-\kappa}(\mathbb{T}^d)$ such that for fixed $t$ the random variable $Y_t$ is distributed as the $\varphi_t$ in \eqref{freeenergy}. We introduce the setting, adapting  \cite{Barashkov_2020} to our Hamiltonian.\\
 Let $\alpha<-\frac{d}{2}$. Let us recall that we call a process $X$ on a probability space $(\Omega,\mathscr{B},\mathbb{P})$ an $\dot{H}^\alpha(\mathbb{T}^d)$-valued cylindrical Brownian motion if there exist $\C$-valued independent Brownian motions $(B^n)_n$ and an orthonormal basis $(e_n)_n$  in $\dot{L}^2(\mathbb{T}^d)$  such that 
 \begin{align*}
     X_t=\sum_n e_n\ B^n_t.
 \end{align*}
We note that due to our choice of $\alpha<-\frac{d}{2}$ this converges in $\dot{H}^\alpha$ since
    $\sum_{n}|e_n|_{H^\alpha}^2 <\infty$. On the right-hand side only the Brownian motions depend on $\omega\in\Omega$.
\\\\
Here, we take $\Omega=C(\R^+;\dot{H}^\alpha(\mathbb{T}^d))$, $\mathscr{B}$ the Borel $\sigma$-algebra on $\Omega$,  and $\mathbb{P}$ the Wiener measure such that the canonical process $X$ is of the form
\begin{align*}
    X_t:\mathbb{T}^d\to\R,\quad x\mapsto\sum_{n\in \Z^d\setminus\{0\}}e^{i\langle x,n\rangle}B_t^n
\end{align*}
for a collection of $\C$-valued Brownian motions $(B_t^n)_{n\in \Z^d\setminus\{0\}}$ with $\Bar{B}_t^{n}=B_t^{-n}$, such that  $B_t^n,B_t^m$ are independent for all $m\neq n$.\\
Next, we introduce a notion of integration with respect to this process. For that we define $\zeta\in C^\infty(\R_+,[0,1])$ as a decreasing function such that $\zeta(s)=1$ for $s\leq \frac{1}{2}$ and $\zeta(s)=0$ for $s\geq 1$. 
Let $\zeta_t(x)\coloneq \zeta(\frac{(1+|x|^2)^\frac{1}{2}}{t})$ and  $\sigma_t(x)\coloneq \left(\frac{d}{dt}(\zeta^2_t(x))\right)^\frac{1}{2}$  for $x\in\R^d$, see Fig. \ref{fig:zeta} and \ref{fig:sigma}.
\begin{figure}[h]
    \centering
    \begin{minipage}{0.45\textwidth}
        \centering
        \includegraphics[width=\linewidth]{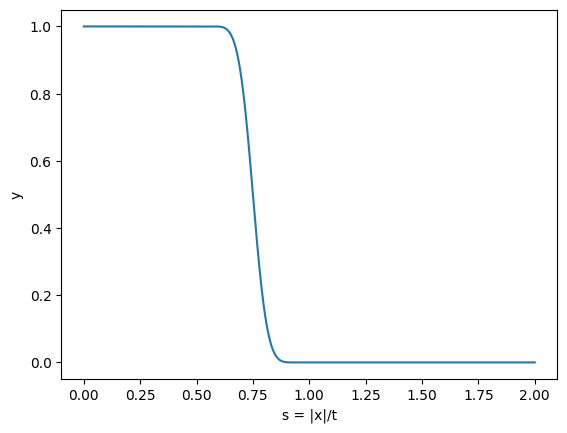}
        \caption{Function $\zeta_t$}
        \label{fig:zeta}
    \end{minipage}
    \hfill
    \begin{minipage}{0.45\textwidth}
        \centering
        \includegraphics[width=\linewidth]{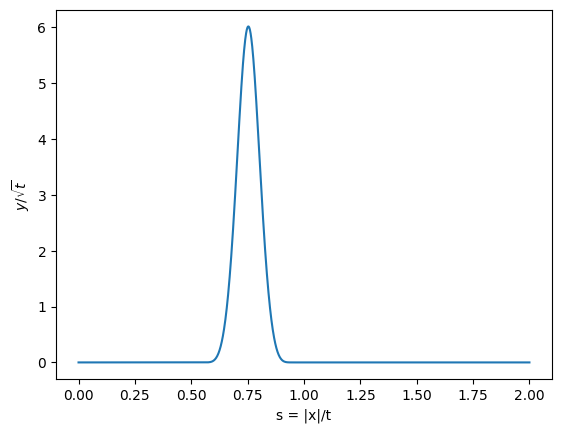}
        \caption{Function $\sigma_t$}
        \label{fig:sigma}
    \end{minipage}
\end{figure}\\
Using the notation \begin{eqnarray}\label{eq:Jt}
    J_t(f)\coloneq \frac{\sigma_t(D)}{(|D|^2+|D|^4)^\frac{1}{2}}f
\end{eqnarray} for a function $f$ we define the centered Gaussian process
\begin{align}\label{eq:Y}
    Y_t&\coloneq\int_0^tJ_s\ dX_s \coloneq \sum_{n\in \Z^d\setminus\{0\}}\int_0^t\frac{\sigma_s(n)}{(|n|^2+|n|^4)^\frac{1}{2}}e^{i\langle \cdot,n\rangle}\ dB_s^n.
\end{align}
Note that this is also well-defined as an $\dot{H}^\alpha$-valued process.\\\\
It remains to show that the values of $Y_t$ are even in $\mathscr{C}^{\frac{4-d}{2}-\kappa}(\T^d)$ for any $\kappa>0$. We use the Besov embedding $B_{p,p}^{\frac{4-d}{2}}(\T^d)\hookrightarrow B_{\infty,\infty}^{\frac{4-d}{2}-\frac{d}{p}}(\T^d)=\mathscr{C}^{\frac{4-d}{2}-\frac{d}{p}}(\T^d)$. Since $Y_t$ is a centered Gaussian process, we have by It\^{o}'s isometry (see \cite[Proposition 2.15]{goodair}) and boundedness of $\zeta$, for all $p\in[1,\infty)$ and all $\beta<\frac{4-d}{2}$
\begin{align*}
    \mathbb{E}\left[\Vert Y_t\Vert_{B_{p,p}^\beta}^p\right]=&\sum_j2^{j\beta p}\mathbb{E}\left[ \Vert \Delta_jY_t\Vert^p_{L^p} \right]
    \sim \sum_j2^{j\beta p}\int_{\mathbb{T}^d}\mathbb{E}\left[ |\Delta_jY_t(x)|^2 \right]^\frac{p}{2}\ dx\\
    =&\sum_j2^{j\beta p}\int_{\mathbb{T}^d}\mathbb{E}\left[ \Big|\Delta_j\sum_{n\in\Z^d\setminus\{0\}}\int_0^t\frac{\sigma_s(n)}{(|n|^2+|n|^4)^\frac{1}{2}}e^{i\langle x,n\rangle}\,dB_t^n\Big|^2 \right]^\frac{p}{2}\ dx\\
    \lesssim& \sum_j2^{j\beta p}\left(\sum_{n\in\Z^d\setminus\{0\}}\frac{\rho^2_j(n)}{|n|^2+|n|^4}\right)^{\frac{p}{2}}
    \lesssim \sum_j2^{j\beta p} 2^{jp\frac{d-4}{2}}<\infty.
\end{align*}
Letting $p\to\infty$, we obtain the claimed H\"older regularity $\mathscr{C}^{\frac{4-d}{2}-\kappa}(\T^d)$. \\\\
%The process $Y$ is a centered Gaussian process. 
One would expect that $Y_t$ is normally distributed with expectation $0$ and variance $\langle\zeta_t(D)\varphi,(-\Delta+\Delta^2)^{-1}\zeta_t(D)\varphi\rangle_{L^2}$, which also describes the measure $(\zeta_t(D))_*\theta$. We verify this by computing the covariance, i.e., we have for any $\varphi,\psi\in\mathcal{S}({\mathbb{T}^d})$ with $\hat{\varphi}(n)\coloneq\int_{\mathbb{T}^d}e^{i\langle n,x\rangle}\varphi(x)\ dx$

\begin{align}\label{eq:var}
    &\mathbb{E}[\langle Y_t,\varphi\rangle_{L^2}\overline{\langle Y_s,\psi\rangle_{L^2}}]= \int_\Omega\left(\int_{\mathbb{T}^d} Y_t(\omega)(\lambda)\overline{\varphi(\lambda)} \ d\lambda \overline{\int_{\mathbb{T}^d} Y_s(\omega)(\lambda)\overline{\psi(\lambda)} \ d\lambda}\right)  \mathbb{P}(d\omega)\nonumber\\
    &=\int_\Omega\left(\int_{\mathbb{T}^d} Y_t(\omega)(\lambda)\varphi(\lambda) \ d\lambda \int_{\mathbb{T}^d} \overline{Y_s(\omega)(\lambda)}\psi(\lambda) \ d\lambda\right)  \mathbb{P}(d\omega)\nonumber\\
    &= \int_\Omega\sum_{n,m\in\Z^d\setminus\{0\}}\left(\int_0^t\frac{\sigma_u(n)}{(|n|^2+|n|^4)^\frac{1}{2}} \hat{\varphi}(n)\ dB_u^n(\omega) \right)\left( \int_0^s\frac{\sigma_u(m)}{(|m|^2+|m|^4)^\frac{1}{2}} \hat{\psi}(m)\ dB_u^{-m}(\omega) \right)\mathbb{P}(d\omega)\nonumber\\
    &= \sum_{n\in\Z^d\setminus\{0\}}\int_\Omega \int_0^{\min(s,t)}\frac{\sigma^2_u(n)}{|n|^2+|n|^4}\ du \ \mathbb{P}(d\omega)\hat{\varphi}(n)\overline{\hat{\psi}(n)}
    =\sum_{n\in\Z^d\setminus\{0\}} \frac{\zeta^2_{\min(s,t)}(n)}{|n|^2+|n|^4}\hat{\varphi}(n)\overline{\hat{\psi}(n)}.
\end{align}

By Plancherel's theorem and since $\lim\limits_{s\to\infty}\zeta_s(n)=1$, \eqref{eq:var} yields the expected variance in the limit $s\to\infty$.

\subsection{Outline of the paper}

The rest of the paper is structured as follows. We first focus on the Hamiltonian \eqref{eq:firstHamiltonian}. In Section \ref{sec:variationalform} we introduce the variational formulation for \eqref{freeenergy} putting this in a more general context within the framework of convex analysis. In Section \ref{sec:renormalization} we deal with the renormalization in two dimensions. We comment on the $\Gamma$-convergence in Section \ref{sec:gamma}. Section \ref{sec:AvilesGiga} connects the preceding theory with the Hamiltonian \eqref{eq:H} related to the Aviles-Giga functional. In Section \ref{sec:extension} we extend the theory to further related Hamiltonians in two and three dimensions.

\section{Variational Formula}\label{sec:variationalform}
In this section we obtain for all $f\in C(\mathscr{C}^{1-\kappa}({\mathbb{T}^d});\R)$ with linear growth a variational formula for
\begin{align*}
    W_T(f)\coloneq - \log\int_{\mathcal{S}'}\exp(-V^f_T(\varphi_T))\ \theta(d\varphi)
    =-\log\mathbb{E}\left[\exp(-V_T^f(Y_T)) \right]
\end{align*}
with $V^f_T(\varphi_T)\coloneq f(\varphi_T)+\int_{\mathbb{T}^d}(|\nabla\varphi_T(x)|^4- \mathcal{O}_T(|\nabla\varphi_T(x)|^2))\ dx$.\\\\ %\BZcom{Notation in Einleitung erkl\"aren}
We first consider the well-known equality (see e.g. \cite{alma991074785089705501})
\begin{align}\label{fenchel}
    -\log\left(\int_X\exp(-f(x))\ \mu(dx)\right) = \inf_{\gamma\in P_\mu(X)} \left\{\int_Xf(x)\ \gamma(dx)+\int_X\log\left(\frac{d\gamma}{d\mu}(x)\right)\ \gamma(dx)\right\}.
\end{align}
In this expression, $X$ is a measure space with the Borel $\sigma$-algebra, $\mu$ is a probability measure, $P_\mu(X)$ is the set of probability measures on $X$ that are equivalent to $\mu$, and $f$ is a bounded measurable function, and we denote by $\frac{d\gamma}{d\mu}$ the Radon-Nikodym density.\\
We will discuss an alternative view on formula \eqref{fenchel}, interpreting it in the context of convex analysis. 
Recall that the Fenchel conjugate of a function $F:N\to\overline{\R}$ on a normed space $N$ is defined as $F^\ast:N'\to\overline{\R}$, $F^\ast(\ell):=\sup_{x\in N}\left(\langle \ell,x\rangle -F(x)\right)$, where $N'$ denotes the dual space to $N$. Since $N$ is canonically embedded into its bidual space, with a slight abuse of notation, we can consider the restriction of the Fenchel biconjugate $F^{\ast\ast}:=(F^\ast)^\ast$ to $N$. For a convex and lower semi-continuous function $F$ it is well known that the function coincides with its Fenchel biconjugate on $N$.\\\
In the next lemma we collect some well-known preliminary results.
\begin{lem}\label{lem:F*}
  Let $\mu$ be a probability measure on $X$. Then the functional $F:L^\infty(X)\to \R$, $f\mapsto \log\left(\int_X\exp(f(x))\ \mu(dx)\right)$ is continuous, convex and for $\gamma\in P_\mu(X)\subset (L^\infty(X))'$ the Fenchel conjugate is of the form
  \begin{align*}
      F^\ast(\gamma)=\int_X\log\left(\frac{d\gamma}{d\mu}(x)\right)\ \gamma(dx).
  \end{align*}
\end{lem}
A proof can be found in \cite[Lemma 3.2.13]{deuschel2001large}, there the functional $F$ is defined on bounded continuous functions.  For the readers convenience we include a proof in the Appendix, see Lemma \ref{lem:F*proof}.\\\\
In the next theorem we characterize the above functional as its Fenchel biconjugate $F^{\ast\ast}$ and give an explicit formula for $F^{\ast\ast}$.
\begin{thm}
    Let $\mu$ be a probability measure on $X$ and $F:L^\infty(X)\to \R$,  $f\mapsto \log\left(\int_X\exp(f(x))\ \mu(dx)\right)$. Then
    \begin{align*}
        F(f)=F^{**}(f)=\sup_{\gamma\in P_\mu(X)}\left\{\int_X f(x)\ \gamma(dx) -\int_X\log\left(\frac{d\mu}{d\gamma}(x)\right)\ \gamma(dx)\right\},
    \end{align*}
    where $F^{**}$ denotes the Fenchel biconjugate of $F$. 
    In particular, it holds
    \begin{align*}
         -\log\left(\int_X\exp(-f(x))\ \mu(dx)\right) = \inf_{\gamma\in P_\mu(X)} \left\{\int_Xf(x)\ \gamma(dx)+\int_X\log\left(\frac{d\gamma}{d\mu}(x)\right)\ \gamma(dx)\right\}.
    \end{align*}
\end{thm}
\begin{proof}
Let $f\in L^\infty(X)$ be arbitrary. Since by Lemma \ref{lem:F*}, $F$ is continuous and convex, we have that $F(f)=F^{\ast\ast}(f)$. We start by deriving a lower bound for $F^{\ast\ast}(f)$ using that the space of probability measures is a subset of $(L^\infty(X))'$. Using Lemma \ref{lem:F*}, in particular that 
\begin{eqnarray*}
F^\ast(\gamma)=\int_X\log\left(\frac{d\gamma}{d\mu}(x)\right)\ \gamma(dx)
\end{eqnarray*}
for $\gamma\in P_\mu(X)$, we obtain 
\begin{align*}
   F(f)&= F^{**}(f)=\sup_{\ell\in(L^\infty(X))'}\left(\langle \ell,f\rangle-F^\ast(\ell)\right)
   \geq \sup_{\gamma\in P_\mu(X)}\left(\langle \gamma,f\rangle-F^\ast(\gamma)\right)\\     &= \sup_{\gamma\in P_\mu(X)}\left\{\int_X f(x)\ \gamma(dx) -\int_X\log\left(\frac{d\mu}{d\gamma}(x)\right)\ \gamma(dx)\right\}.
\end{align*}
To show the other inequality, we consider the measure $\tilde{\gamma}=\frac{1}{\int_X\exp(f(x))\ \mu(dx)}\exp(f)\ \mu\in P_\mu(X)$ and compute
\begin{align*}
    &\int_Xf(y)\ \tilde\gamma(dy)-\int_X\log\left(\frac{d\tilde\gamma}{d\mu}(y)\right)\ \tilde\gamma(dy)\\
    &= \int_Xf(y)\ \tilde\gamma(dy)-\int_X\log\left(\frac{1}{\int_X\exp(-f(x))\ \mu(dx)}\exp(f(y))\right)\ \tilde\gamma(dy)\\
    &=\int_Xf(y)\ \tilde\gamma(dy)+\int_X\log\left(\int_X\exp(f(x))\ \mu(dx)\right)\ \tilde\gamma(dy)-\int_X\log\left(\exp(f(y))\right)\ \tilde\gamma(dy)\\
    &=\log\left(\int_X\exp(f(x))\ \mu(dx)\right).
\end{align*}
Finally, this yields
\begin{align*}
    -\log\left(\int_X\exp(-f(x))\ \mu(dx)\right) &=-F(-f)=-F^{**}(-f)\\
    &= -\sup_{\gamma\in P_\mu(X)}\left\{-\int_X f(x)\ \gamma(dx) -\int_X\log\left(\frac{d\mu}{d\gamma}(x)\right)\ \gamma(dx)\right\}
    \\
    &= \inf_{\gamma\in P_\mu(X)} \left\{\int_Xf(x)\ \gamma(dx)+\int_X\log\left(\frac{d\gamma}{d\mu}(x)\right)\ \gamma(dx)\right\}.
\end{align*}
\end{proof}
In the sequel, $X$ is a probability space that carries a cylindrical Brownian motion $W$ with values in a Hilbert space $H$ and $\mu$ is the Wiener measure. Moreover, for a bounded continuous function $g:C([0,\infty),H)\to \R$ we let $g\circ W$ play the role of the function $f$ in \eqref{fenchel}.\\
In this setting, we obtain the following simplified form for the variational problem. 

\begin{prop}\label{prop:Boue-dupuis}
    Let $g:C([0,\infty),H)\to\R$ be such that \[\int_X(|g(W(x))|+1)\exp(-g(W(x))\ \mu(dx)<\infty.\] Then 
    \begin{align*}
    &-\log\int_X \exp(-g(W(x)))\ \mu(dx)\\
    &=\inf_{v}\left\{\int_X\left(g\left(W(x)+\int_0^{\cdot}v_s(x)\ ds\right)+ \frac{1}{2}\int_0^\infty \Vert v_s(x)\Vert^2\ ds\right) \mu(dx)\right\},
\end{align*}
where the infimum is taken over drifts $v$ such that $v(x)\in L^2([0,\infty),H)$ for almost every $x\in X$. 
\end{prop}
For the readers' convenience, we sketch the main ideas of the proof here, and refer to  \cite{10.1214/aop/1022855876}, \cite{USTUNEL20143058} and \cite{10.1007/BFb0005070} for details.
Starting point is the observation 
that the infimum over $P_\mu$ can be replaced by the infimum over only measures of the form $\frac{d\gamma_v}{d\mu}=\exp\left(\int_0^\infty v_s\ dW_s-\frac{1}{2}\int_0^\infty \Vert v_s\Vert^2\ ds\right)$ for a drift $v$ such that $v(x)\in L^2([0,\infty),H)$ for almost every $x\in X$. Then
\begin{align*}
    &-\log\int_X \exp(-g(W(x)))\ \mu(dx)\\
    &=\inf_{v}\left\{\int_Xg(W(x))\ \gamma_v(dx)+ \int_X\log\left(\frac{d\gamma_v}{d\mu}(x)\right)\ \gamma_v(dx)\right\}\\
    &=\inf_{v}\left\{\int_Xg(W(x))\ \gamma_v(dx)+ \int_X\left(\int_0^\infty v_s(x)\ dW_s(x)-\frac{1}{2}\int_0^\infty \Vert v_s(x)\Vert^2\ ds\right)\ \gamma_v(dx)\right\}\\
    &=\inf_{v}\left\{\int_X\left(g(W(x))\ + \frac{1}{2}\int_0^\infty \Vert v_s(x)\Vert^2\ ds\right) \gamma_v(dx)\right\}.
\end{align*}
In the last line we used
\begin{align*}
    &\int_0^\infty v_s\ dW_s-\frac{1}{2}\int_0^\infty \Vert v_s\Vert^2\ ds\\
    &=\int_0^\infty v_s\ d(W_s-\int_0^sv_l\ dl+\int_0^sv_l\ dl)-\frac{1}{2}\int_0^\infty \Vert v_s\Vert^2\ ds\\
    &=\int_0^\infty v_s\ d(W_s-\int_0^sv_l\ dl)+\frac{1}{2}\int_0^\infty \Vert v_s\Vert^2\ ds
\end{align*}
and the fact that the first integral is by Girsanov's theorem a martingale with expectation zero.

\begin{rmk}
    Positive polynomials satisfy the conditions for $g$ in Proposition \ref{prop:Boue-dupuis}.
\end{rmk}
%\BZ{Rem: Polynomials valid}
Finally, we apply Proposition \ref{prop:Boue-dupuis} to \eqref{freeenergy} with $g=f(\int_0^TJ_sd(\cdot)_s)+V_T(\int_0^TJ_sd(\cdot)_s)$, $H=\dot{L}^2(\mathbb{T}^d)$ and $X=\Omega$. Note that this is possible since $g$ has only polynomial growth. This yields
\begin{align}\label{eq:freeenergy}
    W_T(f)=- \log \mathbb{E}[\exp(-V_T^f(Y_T))]=\inf_{v\in\mathbb{H}}\mathbb{E}\left[ V^f_T(Y_T+Z_T(v))+\frac{1}{2}\int_0^\infty\Vert v_s\Vert^2_{L^2}\ ds\right],
\end{align}
where 
\begin{align*}
    Z_T(u)&\coloneq \int_0^T\frac{\sigma_t(D)}{(|D|^2+|D|^4)^\frac{1}{2}}u_t\ dt
\end{align*}
and $\mathbb{H}$ is the space of $\dot{L}^2(\mathbb{T}^d)$-valued  predictable processes, whose trajectories lie almost surely in $\mathcal{H}\coloneq L^2([0,\infty);\dot{L}^2( \mathbb{T}^d))$. 
In the following we write 
\begin{align}\label{eq:FT}
    F_T(v)\coloneq\mathbb{E}\left[ V^f_T(Y_T+Z_T(v))+\frac{1}{2}\int_0^\infty\Vert v_s\Vert^2_{L^2}\ ds\right].
\end{align}

\section{Renormalization}\label{sec:renormalization}

In this section we introduce the renormalization constants for \eqref{eq:firstHamiltonian}. Renormalization is necessary for the convergence of the functionals $(F_T)_T$. Roughly speaking we need to subtract infinities in order to have a convergence of the singular quantity $|\nabla\varphi_T|^4$.\\
We will decompose $F_T$ as given in \eqref{eq:FT} in the form
\begin{align}\label{eq:FT_Phi}
    F_T(u)=\mathbb{E}\left[ f(Y_T+Z_T(u))+\Phi_T(u)+\Vert\nabla Z_T(u)\Vert_{L^4}^4+\frac{1}{2}\int_0^\infty\Vert u_s\Vert^2_{L^2}\ ds\right]. 
\end{align}
To prove $\Gamma$-convergence and equicoercivity of $(F_T)_T$ we need upper and lower bounds for $F_T$. It will turn out that the last two terms on the right handside of \eqref{eq:FT_Phi} behave well under convergence in the chosen weak topology. The terms collected in $\Phi_T$ require more care, and we will derive explicit estimates below.\\
We use Wick-type renormalization (see e.g. \cite[Chapter III]{Janson_1997}) adapted to the vectorial setting, where we subtract the drift parts of powers of $\nabla Y_T$ and keep only the martingale parts to obtain convergence to well-defined distributions.
\begin{lem}\label{Wick-Powers}
The following processes are martingales with respect to the natural filtration of $X$ and the Wiener measure $\mathbb{P}$: 
   \begin{align*}
    [\![\nabla Y_T^4]\!]\coloneq&(\nabla Y_T)^4- 6\langle\partial_1Y,\partial_1Y\rangle_T(\partial_1Y_T)^2-6\langle\partial_2Y,\partial_2Y\rangle_T(\partial_2Y_T)^2-2\langle\partial_2Y,\partial_2Y\rangle_T(\partial_1Y_T)^2\\
    &-2\langle\partial_1Y,\partial_1Y\rangle_T(\partial_2Y_T)^2-4\langle\partial_1Y,\partial_2Y\rangle_T^2+3\langle\partial_1Y,\partial_1Y\rangle_T^2+3\langle\partial_2Y,\partial_2Y\rangle_T^2\\
    &+2\langle\partial_1Y,\partial_1Y\rangle_T\langle\partial_2Y,\partial_2Y\rangle_T\\
    &-8 \int_0^T\int_0^s \partial_1 Y_u\ d(\partial_2 Y_u) \ d\langle\partial_1 Y,\partial_2Y\rangle_s -8 \int_0^T\int_0^s \partial_2 Y_u\ d(\partial_1 Y_u) \ d\langle\partial_1 Y,\partial_2Y\rangle_s\\
    &-8 \int_0^T\int_0^s \partial_1 Y_u\ d\langle\partial_2 Y,\partial_2 Y\rangle_u\ d(\partial_1Y_s)-8 \int_0^T\int_0^s \langle\partial_2 Y,\partial_2 Y\rangle_u\ d(\partial_1 Y_u)\ d(\partial_1Y_s)\\
    &-8 \int_0^T\int_0^s \partial_2 Y_u\ d\langle\partial_1 Y,\partial_1 Y\rangle_u\ d(\partial_2Y_s)-8 \int_0^T\int_0^s \langle\partial_1 Y,\partial_1 Y\rangle_u\ d(\partial_2 Y_u)\ d(\partial_2Y_s),
    \\
    [\![\nabla Y_T^3]\!]\coloneq&(\nabla Y_T)^3-\begin{pmatrix}
        3\langle\partial_1Y,\partial_1Y\rangle_T\partial_1Y_T+2\langle\partial_1Y,\partial_2Y\rangle_T\partial_2Y_T+\langle\partial_2Y,\partial_2Y\rangle_T\partial_1Y_T\\ 
        3\langle\partial_2Y,\partial_2Y\rangle_T\partial_2Y_T+2\langle\partial_1Y,\partial_2Y\rangle_T\partial_1Y_T+\langle\partial_1Y,\partial_1Y\rangle_T\partial_2Y_T
    \end{pmatrix},
    \end{align*} and for all $i,j\in\{1,2\}$
    \begin{align*}
    [\![\partial_i Y_T\partial_j Y_T]\!]\coloneq&\partial_i Y_T\partial_j Y_T-\langle\partial_iY,\partial_jY\rangle_T.
\end{align*}
\end{lem}
\begin{proof}
 Note, that the martingale property is preserved if we integrate a martingale with respect to a martingale. Itô's formula yields for $i,j\in\{1,2\}$
\begin{align*}
    \partial_i Y_T\partial_j Y_T= \int_0^T\partial_i Y_s\ d(\partial_j Y_s)+\int_0^T\partial_j Y_s\ d(\partial_i Y_s)+\langle\partial_iY,\partial_jY\rangle_T.
\end{align*}
For the cubic term we have for $j,i\in\{1,2\}$ and $i\ne j$
\begin{align*}
    ((\nabla Y_T)^3)_i&=(\nabla Y_T)^2\partial_i Y_T\\
    &=6\int_0^T\int_0^s\partial_i Y_u\ d(\partial_i Y_u)\ d(\partial_i Y_s)+2\int_0^T\int_0^s\partial_j Y_u\ d(\partial_j Y_u)\ d(\partial_i Y_s)\\
    &+2\int_0^T\int_0^s\partial_j Y_u\ d(\partial_i Y_u)\ d(\partial_j Y_s)+2\int_0^T\int_0^s\partial_i Y_u\ d(\partial_j Y_u)\ d(\partial_j Y_s)\\
    &+3\int_0^T \partial_i Y_s\ d\langle\partial_i Y,\partial_i Y\rangle_s+3\int_0^T\int_0^s d\langle\partial_i Y,\partial_i Y\rangle_u\ d(\partial_i Y_s)\\
    &+2\int_0^T \partial_j Y_s\ d\langle\partial_i Y,\partial_j Y\rangle_s+2\int_0^T\int_0^s d\langle\partial_i Y,\partial_j Y\rangle_u\ d(\partial_j Y_s)\\
    &+\int_0^T \partial_i Y_s\ d\langle\partial_j Y,\partial_j Y\rangle_s+\int_0^T\int_0^s d\langle\partial_j Y,\partial_j Y\rangle_u\ d(\partial_i Y_s).
\end{align*}
Since the covariation processes are of bounded variation we obtain with Itô's product rule 
\begin{align*}
3\int_0^T \partial_i Y_s\ d\langle\partial_i Y,\partial_i Y\rangle_s+3\int_0^T\int_0^s d\langle\partial_i Y,\partial_i Y\rangle_u\ d(\partial_i Y_s)\\
    +2\int_0^T \partial_j Y_s\ d\langle\partial_i Y,\partial_j Y\rangle_s+2\int_0^T\int_0^s d\langle\partial_i Y,\partial_j Y\rangle_u\ d(\partial_j Y_s)\\
    +\int_0^T \partial_i Y_s\ d\langle\partial_j Y,\partial_j Y\rangle_s+\int_0^T\int_0^s d\langle\partial_j Y,\partial_j Y\rangle_u\ d(\partial_i Y_s)\\
    =3\langle\partial_iY,\partial_iY\rangle_T\partial_iY_T+2\langle\partial_iY,\partial_jY\rangle_T\partial_jY_T+\langle\partial_jY,\partial_jY\rangle_T\partial_iY_T.
\end{align*}
Lastly, we observe the fourth power term,
\begin{align*}
    (\nabla Y_T)^4= (\partial_1Y_T)^4+2(\partial_1Y_T)^2(\partial_2Y_T)^2+(\partial_2Y_T)^4\\
    =24\int_0^T\int_0^s\int_0^u\partial_1Y_l\ d(\partial_1Y_l)\ d(\partial_1Y_u)\ d(\partial_1Y_s)+8\int_0^T\int_0^s\int_0^u\partial_2Y_l\ d(\partial_1Y_l)\ d(\partial_2Y_u)\ d(\partial_1Y_s)\\
    +8\int_0^T\int_0^s\int_0^u\partial_2Y_l\ d(\partial_2Y_l)\ d(\partial_1Y_u)\ d(\partial_1Y_s)+8\int_0^T\int_0^s\int_0^u\partial_1Y_l\ d(\partial_2Y_l)\ d(\partial_2Y_u)\ d(\partial_1Y_s)\\
    +24\int_0^T\int_0^s\int_0^u\partial_2Y_l\ d(\partial_2Y_l)\ d(\partial_2Y_u)\ d(\partial_2Y_s)+8\int_0^T\int_0^s\int_0^u\partial_1Y_l\ d(\partial_2Y_l)\ d(\partial_1Y_u)\ d(\partial_2Y_s)\\
    +8\int_0^T\int_0^s\int_0^u\partial_1Y_l\ d(\partial_1Y_l)\ d(\partial_2Y_u)\ d(\partial_2Y_s)+8\int_0^T\int_0^s\int_0^u\partial_2Y_l\ d(\partial_1Y_l)\ d(\partial_1Y_u)\ d(\partial_2Y_s)\\
    +12 \int_0^T\int_0^s \partial_1 Y_u\ d\langle\partial_1 Y,\partial_1Y\rangle_u\ d(\partial_1Y_s)+12 \int_0^T\int_0^s \partial_1 Y_u\ d\langle\partial_2 Y,\partial_2 Y\rangle_u\ d(\partial_1Y_s)\\
    +12 \int_0^T\int_0^s \langle\partial_1 Y,\partial_1Y\rangle_u\ d(\partial_1 Y_u)\ d(\partial_1Y_s)+12 \int_0^T\int_0^s \langle\partial_2 Y,\partial_2 Y\rangle_u\ d(\partial_1 Y_u)\ d(\partial_1Y_s)\\
   +12 \int_0^T\int_0^s \partial_2 Y_u\ d\langle\partial_2 Y,\partial_2Y\rangle_u\ d(\partial_2Y_s)+12 \int_0^T\int_0^s \partial_2 Y_u\ d\langle\partial_1 Y,\partial_1 Y\rangle_u\ d(\partial_2Y_s)\\
    +12 \int_0^T\int_0^s \langle\partial_2 Y,\partial_2Y\rangle_u\ d(\partial_2 Y_u)\ d(\partial_2Y_s)+12 \int_0^T\int_0^s \langle\partial_1 Y,\partial_1 Y\rangle_u\ d(\partial_2 Y_u)\ d(\partial_2Y_s)\\
    +12 \int_0^T\int_0^s \partial_1 Y_u\ d(\partial_1 Y_u) \ d\langle\partial_1 Y,\partial_1Y\rangle_s+ 6 \int_0^T\int_0^s d\langle\partial_1 Y,\partial_1 Y\rangle_u \ d\langle\partial_1 Y,\partial_1Y\rangle_s\\
    +4 \int_0^T\int_0^s \partial_2 Y_u\ d(\partial_2 Y_u) \ d\langle\partial_1 Y,\partial_1Y\rangle_s+ 2 \int_0^T\int_0^s d\langle\partial_1 Y,\partial_1Y\rangle_u \ d\langle\partial_2 Y,\partial_2Y\rangle_S\\
    +12 \int_0^T\int_0^s \partial_2 Y_u\ d(\partial_2 Y_u) \ d\langle\partial_2 Y,\partial_2Y\rangle_s+ 6 \int_0^T\int_0^s d\langle\partial_2 Y,\partial_2Y\rangle_u \ d\langle\partial_2 Y,\partial_2Y\rangle_s\\
    +4 \int_0^T\int_0^s \partial_1 Y_u\ d(\partial_1 Y_u) \ d\langle\partial_2 Y,\partial_2Y\rangle_s+ 2 \int_0^T\int_0^s d\langle\partial_2 Y,\partial_2Y\rangle_u \ d\langle\partial_1 Y,\partial_1Y\rangle_s\\
    +8 \int_0^T\int_0^s \partial_1 Y_u\ d(\partial_2 Y_u) \ d\langle\partial_1 Y,\partial_2Y\rangle_s +8 \int_0^T\int_0^s \partial_2 Y_u\ d(\partial_1 Y_u) \ d\langle\partial_1 Y,\partial_2Y\rangle_s\\
    + 8 \int_0^T\int_0^s d\langle\partial_1 Y,\partial_2Y\rangle_u \ d\langle\partial_1 Y,\partial_2Y\rangle_s.
\end{align*}
As above the drift part can be rewritten as 
\begin{align*}
  &6\langle\partial_1Y,\partial_1Y\rangle_T(\partial_1Y_T)^2+6\langle\partial_2Y,\partial_2Y\rangle_T(\partial_2Y_T)^2+2\langle\partial_2Y,\partial_2Y\rangle_T(\partial_1Y_T)^2\\
    &+2\langle\partial_1Y,\partial_1Y\rangle_T(\partial_2Y_T)^2+4\langle\partial_1Y,\partial_2Y\rangle_T^2-3\langle\partial_1Y,\partial_1Y\rangle_T^2-3\langle\partial_2Y,\partial_2Y\rangle_T^2\\
    &-2\langle\partial_1Y,\partial_1Y\rangle_T\langle\partial_2Y,\partial_2Y\rangle_T\\
    &+8 \int_0^T\int_0^s \partial_1 Y_u\ d(\partial_2 Y_u) \ d\langle\partial_1 Y,\partial_2Y\rangle_s +8 \int_0^T\int_0^s \partial_2 Y_u\ d(\partial_1 Y_u) \ d\langle\partial_1 Y,\partial_2Y\rangle_s\\
    &+8 \int_0^T\int_0^s \partial_1 Y_u\ d\langle\partial_2 Y,\partial_2 Y\rangle_u\ d(\partial_1Y_s)+8 \int_0^T\int_0^s \langle\partial_2 Y,\partial_2 Y\rangle_u\ d(\partial_1 Y_u)\ d(\partial_1Y_s)\\
    &+8 \int_0^T\int_0^s \partial_2 Y_u\ d\langle\partial_1 Y,\partial_1 Y\rangle_u\ d(\partial_2Y_s)+8 \int_0^T\int_0^s \langle\partial_1 Y,\partial_1 Y\rangle_u\ d(\partial_2 Y_u)\ d(\partial_2Y_s).
\end{align*}
These processes are in $L^1(\Omega)$ since powers of Gaussian processes are integrable by the exponential decay of these processes. 
\end{proof}
\begin{rmk}
    The terms $\langle\partial_iY,\partial_jY\rangle_t=\sum_{n\in\Z^2\setminus\{0\}}\frac{n_i(-n_j)}{|n|^2+|n|^4}\zeta_t^2(n)$ are constant in  $x$ and $\omega$.
\end{rmk}
This yields the following renormalization.

\begin{thm}\label{thm:renormalization}
    
Choosing 
\begin{align*}
     a_T=&6\langle\partial_1Y,\partial_1Y\rangle_T(\partial_1Y_T)^2+6\langle\partial_2Y,\partial_2Y\rangle_T(\partial_2Y_T)^2+2\langle\partial_2Y,\partial_2Y\rangle_T(\partial_1Y_T)^2\\
    &+2\langle\partial_1Y,\partial_1Y\rangle_T(\partial_2Y_T)^2+4\langle\partial_1Y,\partial_2Y\rangle_T^2-3\langle\partial_1Y,\partial_1Y\rangle_T^2-3\langle\partial_2Y,\partial_2Y\rangle_T^2\\
    &-2\langle\partial_1Y,\partial_1Y\rangle_T\langle\partial_2Y,\partial_2Y\rangle_T\\
    &+8 \int_0^T\int_0^s \partial_1 Y_u\ d(\partial_2 Y_u) \ d\langle\partial_1 Y,\partial_2Y\rangle_s +8 \int_0^T\int_0^s \partial_2 Y_u\ d(\partial_1 Y_u) \ d\langle\partial_1 Y,\partial_2Y\rangle_s\\
    &+8 \int_0^T\int_0^s \partial_1 Y_u\ d\langle\partial_2 Y,\partial_2 Y\rangle_u\ d(\partial_1Y_s)+8 \int_0^T\int_0^s \langle\partial_2 Y,\partial_2 Y\rangle_u\ d(\partial_1 Y_u)\ d(\partial_1Y_s)\\
    &+8 \int_0^T\int_0^s \partial_2 Y_u\ d\langle\partial_1 Y,\partial_1 Y\rangle_u\ d(\partial_2Y_s)+8 \int_0^T\int_0^s \langle\partial_1 Y,\partial_1 Y\rangle_u\ d(\partial_2 Y_u)\ d(\partial_2Y_s),\\
     b_T=&\begin{pmatrix}
        3\langle\partial_1Y,\partial_1Y\rangle_T\partial_1Y_T+2\langle\partial_1Y,\partial_2Y\rangle_T\partial_2Y_T+\langle\partial_2Y,\partial_2Y\rangle_T\partial_1Y_T\\ 
        3\langle\partial_2Y,\partial_2Y\rangle_T\partial_2Y_T+2\langle\partial_1Y,\partial_2Y\rangle_T\partial_1Y_T+\langle\partial_1Y,\partial_1Y\rangle_T\partial_2Y_T
    \end{pmatrix},\text{\ and}  
\end{align*}
\begin{align*}
     &c_T:\R^2\to\R,\quad c_T((x_1,x_2))=&&6\langle\partial_1Y,\partial_1Y\rangle_Tx_1^2+6\langle\partial_2Y,\partial_2Y\rangle_Tx_2^2+2\langle\partial_1Y,\partial_1Y\rangle_Tx_2^2\\
     & &&+2\langle\partial_2Y,\partial_2Y\rangle_Tx_1^2+8\langle\partial_1Y,\partial_2Y\rangle_Tx_1x_2
     \end{align*}
yields
\begin{align*}
  &  \int|\nabla Y_T+\nabla Z_T(u)|^4-c_T(\nabla Z_T(u))-4b_T\cdot\nabla Z_T(u)-a_T\\
  &  =\int[\![\nabla Y_T^4]\!]
    +4[\![\nabla Y_T^3]\!]\cdot\nabla Z_T(u)+4\nabla Y_T\cdot(\nabla Z_T(u))^3
    +(\nabla Z_T(u))^4\\&
    +\sum_{i,j,k,l\in\{1,2\}}2(\delta_{ik}\delta_{jl}+\delta_{il}\delta_{jk}+\delta_{ij}\delta_{kl})[\![\partial_iY_T\partial_j Y_T]\!]\partial_k Z_T(u) \partial_l Z_T(u).
\end{align*}
If we define
\begin{align}\label{eq:PhiT}
    &\Phi_T(u)\coloneq\int4\nabla Y_T\cdot(\nabla Z_T(u))^3 +4[\![\nabla Y_T^3]\!]\cdot\nabla Z_T(u)\nonumber\\
    &+\sum_{i,j,k,l\in\{1,2\}}2(\delta_{ik}\delta_{jl}+\delta_{il}\delta_{jk}+\delta_{ij}\delta_{kl})[\![\partial_iY_T\partial_j Y_T]\!]\partial_k Z_T(u) \partial_l Z_T(u)
\end{align}
we obtain the representation
\begin{align*}
    F_T(u)=\mathbb{E}\left[ f(Y_T+Z_T(u))+\Phi_T(u)+\Vert\nabla Z_T(u)\Vert_{L^4}^4+\frac{1}{2}\int_0^\infty\Vert u_s\Vert^2_{L^2}\ ds\right].
\end{align*}
\end{thm}
\begin{proof}
    Having
    \begin{align*}
        &|\nabla Y_T+\nabla Z_T(u)|^4\\&=(\nabla Y)^4+4(\nabla Y)^3\cdot\nabla Z_T(u)+2(\nabla Y_T)^2(\nabla Z_T(u))^2+4(\nabla Y_T\cdot\nabla Z_T(u))^2\\&+\nabla Y_T\cdot (\nabla Z_T(u))^3+(\nabla Z_T(u))^4\\
        &=(\nabla Y)^4+4(\nabla Y)^3\cdot\nabla Z_T(u) +\nabla Y_T\cdot (\nabla Z_T(u))^3+(\nabla Z_T(u))^4\\&
        +\sum_{i,j,k,l\in\{1,2\}}2(\delta_{ik}\delta_{jl}+\delta_{il}\delta_{jk}+\delta_{ij}\delta_{kl})\partial_iY_T\partial_j Y_T\partial_k Z_T(u) \partial_l Z_T(u)
    \end{align*}
    yields the intended form of $F_T$.
\end{proof}
Finally, we obtain the following bounds.
\begin{lem}\label{lem:lubounds}
    There exist $C>0$ and $0<\delta<1$ such that
    \begin{align}\label{lubounds}
        -C+(1-\delta)\mathbb{E}\left[\Vert \nabla Z_T(u)\Vert_{L^4}^4+\frac{1}{2}\Vert u\Vert_\mathcal{H}^2\right]\leq F_T(u)\leq C+(1+\delta)\mathbb{E}\left[\Vert \nabla Z_T(u)\Vert_{L^4}^4+\frac{1}{2}\Vert u\Vert_\mathcal{H}^2\right].
    \end{align}
\end{lem}
\begin{proof} The proof relies heavily on the estimate $\Vert\nabla Z_T(u)\Vert_{H^1}\lesssim\Vert u\Vert_\mathcal{H}$. This is in the spirit of \cite[Lemma 2]{Barashkov_2020}, and we provide an explicit proof in the appendix (see Section \ref{sec:ZTu}). Then we have by linear growth of $f$
\begin{eqnarray*}
    |f(Y_T+Z_T(u))|&\leq& C\left(\|Y_T\|_{C^{1-\kappa}}+\|Z_T(u)\|_{C^{1-\kappa}}\right)
\end{eqnarray*}
The first  term on the right hand-side is uniformly bounded in $T$, and for the second term, we obtain for all $\delta\in (0,1)$, by Sobolev embedding and Poincar\'{e}'s inequality (recall that $u$ and hence also $Z_T(u)$ have zero mean)
\[\|Z_T(u)\|_{C^{1-\kappa}}\leq C\|Z_T(u)\|_{H^2}\leq C\|\nabla Z_T(u)\|_{H^1}\leq C\|u\|_{\mathcal{H}}\leq \frac{C^2}{\delta}+\delta\|u\|_{\mathcal{H}}^2.\]
It remains to bound $\Phi_T$ (see \eqref{eq:PhiT}).
By Lemma \ref{lem:generalbounds} we can find $0<\delta<1$ such that for the three terms, for all $T$ we have
\begin{align*}
    |\int [\![\nabla Y_T^3]\!]\cdot\nabla Z_T(u)|&\leq C(\delta)\Vert [\![\nabla Y_T^3]\!]\Vert_{H^{-1}}^2+\delta \Vert u\Vert_\mathcal{H}^2 \\
    |\int [\![\partial_iY_T\partial_j Y_T]\!]\partial_k Z_T(u) \partial_l Z_T(u)|&\leq C(\delta)\Vert[\![\partial_iY_T\partial_j Y_T]\!]\Vert_{W^{-\varepsilon,\frac{5}{2}}}^\frac{5}{2}+\delta (\Vert u\Vert^2_\mathcal{H}+\Vert \nabla Z_T\Vert_{L^4}^4)\\
   |\int \nabla Y_T\cdot(\nabla Z_T(u))^3|&\leq C(\delta)\Vert\nabla Y_T\Vert_{W^{-\varepsilon,5}}^{5}+\delta (\Vert u\Vert^2_\mathcal{H}+\Vert \nabla Z_T\Vert_{L^4}^4) .
\end{align*}
By Lemma \ref{WickRegularity} we can bound the norms of the Wick-powers independently of $T$ in expectation. This concludes the proof.
\end{proof}

\section{Gamma-Convergence}\label{sec:gamma}
To obtain the convergence of the free energies (see \eqref{eq:freeenergy})
\begin{align*}
    W_T(f)&=\inf_{u\in\mathbb{H}}F_T(u)=\inf_{u\in\mathbb{H}}\mathbb{E}\left[ f(Y_T+Z_T(u))+\Phi_T(u)+\Vert\nabla Z_T(u)\Vert_{L^4}^4+\frac{1}{2}\int_0^\infty\Vert u_s\Vert^2_{L^2}\ ds\right]
\end{align*}
it would be sufficient to prove $\Gamma$-convergence and equicoercivity of the sequence $(F_T)_T$ (see Theorem \ref{minimizer}). A major technical difficulty here is that the functionals control $\|u\|_{\mathcal{H}}$ only in expectation (see Lemma \ref{lem:lubounds}). In order to still get equicoercivity we need to reformulate the problem and choose a different topological space.
This has been elaborated in \cite{Barashkov_2020} for the $\Phi_4^3$-model. We follow closely their lines of argument and comment only on the differences.\\\\
First, we introduce the space
\begin{align*}
    \mathscr{X}\coloneq C([0,\infty]; \mathscr{C}^{1-\kappa}(\mathbb{T}^2;\R)\times\mathscr{C}^{-\kappa}(\mathbb{T}^2;\R^{2\times 2})\times \mathscr{C}^{-\kappa}(\mathbb{T}^2;\R^2))
\end{align*}
and note that $\mathbb{Y}\coloneq( Y,( [\![\partial_iY_T\partial_jY_T]\!])_{i,j\in \{1,2\}},[\![\nabla Y^3]\!])\in\mathscr{X}$ by Lemma \ref{WickRegularity}.\\\\
We consider the following topological space. Recall that $\mathbb{H}$ is defined below \eqref{eq:freeenergy}.
\begin{defi}\label{def:wekconv}
    Denote the canonical variable on $\mathscr{X}\times\mathcal{H}$ by $(\mathbb{X},u)$ and consider the space of probability measures
    \begin{align*}
        \mathcal{Y}\coloneq \{\mu\in \mathcal{P}(\mathscr{X}\times\mathcal{H})\ |\ \mathbb{E}_\mu[\Vert u\Vert_\mathcal{H}^2]<\infty\}.
    \end{align*}
    We equip this space with the topology induced by the convergence $\mu_n\to\mu$ which is defined as follows:
    \begin{itemize}
        \item $(\mu_n)_{n}$ converges to $\mu$ weakly on $\mathscr{X}\times \mathcal{H}_w$, where $\mathcal{H}_w$ is $\mathcal{H}$ equipped with the weak topology,
        \item $\sup_n\mathbb{E}_{\mu_n}[\Vert u\Vert_\mathcal{H}^2]<\infty$.
    \end{itemize}
    Let 
    \begin{align}\label{eq:X}
        \mathcal{X}\coloneq \{\mu\in\mathcal{Y}\ |\ \mu=\Law_\mathbb{P}(\mathbb{Y},u)\text{ for some }u\in\mathbb{H}\}
    \end{align}
    and denote by $\bar{\mathcal{X}}\subset\mathcal{Y}$ the closure of $\mathcal{X}$ in $\mathcal{Y}$.
\end{defi}
\begin{rmk}\label{law}
    \begin{itemize}
        \item For every $\mu\in\bar{\mathcal{X}}$ we have $\Law_\mu(\mathbb{X})=\Law_\mathbb{P}(\mathbb{Y})$ because  for $f\in C_b(\mathscr{X)}$ and a sequence $(\mu_n)_n\subset\mathcal{X}$ that converges weakly to $\mu$, we have 
    \begin{align*}
        \int_{\Omega} f(\mathbb{Y})\,\mathrm d\mathbb{P}=
    \int_{\mathscr{X}\times \mathcal{H}} f\circ\pi_\mathscr{X} \,\mathrm d\mu_n
    \xrightarrow[]{n\to\infty}
    \int_{\mathscr{X}\times \mathcal{H}} f\circ\pi_\mathscr{X} \,\mathrm d\mu=
    \int_{\Omega} f(\mathbb{Y})\,\mathrm d\mathbb{P}.
    \end{align*}
    \item We take the weak topology on $\mathcal{H}$ to ensure equicoercivity of $(F_T)_T$. 
    \item The second condition in Definition \ref{def:wekconv} is not very restrictive for our application because an asymptotically infimizing sequence satisfies this condition.
    \end{itemize}
     
\end{rmk}
Since for every $u\in\mathbb{H}$ there exists an associated $\mu\in\mathcal{X}$ (see \eqref{eq:X}) we can rewrite the optimization problem \eqref{eq:freeenergy} as 
\begin{align}\label{eq:WinfF}
    W_T(u)=\inf_{\mu\in\mathcal{X}}\tilde{F}_T(\mu),
\end{align}
where
\begin{align*}
    \tilde{F}_t(\mu)&\coloneq \mathbb{E}_\mu\left[f(\mathbb{X}^1_t+Z_t(u))+\Phi_t(\mathbb{X},u)+\Vert \nabla Z_t(u)\Vert_{L^4}^4+\frac{1}{2}\Vert u\Vert_\mathcal{H}^2\right]
\end{align*}
with
\begin{align*}
    \Phi_t(\mathbb{X},u)\coloneq \int4 \nabla\mathbb{X}^1_t\cdot(\nabla Z_t(u))^3 +4\mathbb{X}^3_t\cdot\nabla Z_t(u)\\
    +\sum_{i,j,k,l\in\{1,2\}}2(\delta_{ik}\delta_{jl}+\delta_{il}\delta_{jk}+\delta_{ij}\delta_{kl})(\mathbb{X}^2_t)_{i,j}\partial_k Z_t(u) \partial_l Z_t(u)
\end{align*}
for all $t\in[0,\infty]$, where for $t=\infty$, we take the respective expressions with $T$ replaced by $\infty$.\\\\
Moreover, we extend the infimum in \eqref{eq:WinfF} to the closure of $\mathcal{X}$. This is required for equicoercivity, since we need not only relative compactness, but compactness of the sublevel sets. Note that it is not trivial that the infimum over the extended space is still equal to the infimum over the drifts since we cannot ensure that every measure in $\bar{\mathcal{X}}$ has a distribution that imitates a drift in combination with the Gaussian process $Y$. Nevertheless, the equality of the infima holds.
\begin{thm}\label{inf}
 If $T\in[0,\infty]$ we have
 \begin{align*}
    \inf_{\mu\in\mathcal{X}}\tilde{F}_T(\mu)=\inf_{\mu\in\bar{\mathcal{X}}}\tilde{F}_T(\mu).
 \end{align*}
\end{thm}
    For a proof, we refer to \cite[Lemma 14, Lemma 17]{Barashkov_2020} . Note, that the proof in the two dimensional setting is simpler since we can treat $T=\infty$ as the finite case.\\\\
Having these definitions we continue with stating the results.\\
\begin{thm}\label{EC}
    The family $(\tilde{F}_T)_T$ is equicoercive on $\bar{\mathcal{X}}$.
\end{thm}
For a proof see \cite[Corollary 2]{Barashkov_2020}. By a Prohorov type argument, the set $\mathcal{K}\coloneq \{\mu\in\bar{\mathcal{X}}:\mathbb{E}[\Vert u\Vert_\mathcal{H}^2]\leq K\}$ is compact for every $K\in \R$. Together with the lower bound \eqref{lubounds} containing the norm of $u$ one obtains equicoercivity.\\

\begin{thm}\label{thm:GC}
    The family $(\tilde{F}_T)_T$ $\varGamma$-converges to $\tilde{F}_\infty$ on $\bar{\mathcal{X}}$. 
\end{thm}
 For a proof see \cite[Theorem 6]{Barashkov_2020}. An important step in the proof is to show  that $(\Phi_T(\mathbb{X}^T,u^T))_T$ converges if $(\mathbb{X}^T)_T$ converges in $\mathscr{X}$ and $(u^T)_T$ converges in $\mathcal{H}_w$. This follows by the compactness of $\nabla Z:\mathcal{H}\to C([0,\infty]; W^{\varepsilon,4})$ for small $\varepsilon$, which we show in Lemma \ref{compactZ} below in a more general setting, and the estimates 
\begin{align*}
    |\int [\![\nabla Y_T^3]\!]\cdot\nabla Z_T(u)|&\leq \Vert [\![\nabla Y_T^3]\!]\Vert_{W^{-\varepsilon,\frac{4}{3}}}\Vert \nabla Z_T(u)\Vert_{W^{\varepsilon,4}}\\
    |\int [\![\partial_iY_T\partial_j Y_T]\!]\partial_k Z_T(u) \partial_l Z_T(u)|&\leq \Vert[\![\partial_iY_T\partial_j Y_T]\!]\Vert_{W^{-\varepsilon,\frac{5}{2}}}\Vert (\nabla Z_T(u))^2\Vert_{W^{\varepsilon,\frac{5}{3}}}\\&\leq\Vert [\![\partial_iY_T\partial_j Y_T]\!]\Vert_{W^{-\varepsilon,\frac{5}{2}}}\Vert \nabla Z_T(u)\Vert_{W^{\varepsilon,\frac{20}{7}}}\Vert \nabla Z_T(u)\Vert_{L^{4}}\\
   |\int \nabla Y_T\cdot(\nabla Z_T(u))^3|&\leq \Vert \nabla Y_T\Vert_{W^{-\varepsilon,5}}\Vert (\nabla Z_T(u))^3\Vert_{W^{\varepsilon,\frac{5}{4}}}\\&\leq\Vert \nabla Y_T\Vert_{W^{-\varepsilon,5}}\Vert \nabla Z_T(u)\Vert_{W^{\varepsilon,\frac{10}{3}}}\Vert \nabla Z_T(u)\Vert^2_{L^{4}}.
\end{align*}
These estimates are shown rigorously in Lemma \ref{lem:generalbounds}.
Note that Theorems \ref{EC} and \ref{thm:GC} together yield the convergence of the infima $\inf_{\bar{\mathcal{X}}} \tilde{F}_T$, see Theorem \ref{minimizer}. Finally, to get the existence of a finite limit, one needs to bound $(\inf F_T)_T$.

\begin{thm}\label{constantbounds}
    For all $f\in C(\mathscr{C}^{1-\kappa};\R)$ with linear growth there exists a finite constant $C$ such that
    \begin{align*}
        \sup_T|W_T(f)|\leq C.
    \end{align*}

\end{thm}
\begin{proof}
    The boundedness follows from  \eqref{lubounds} and plugging in $u\equiv 0$.
\end{proof}
Summarizing, we conclude with the following result.
\begin{thm}\label{maintheorem}
We have
\begin{align*}
    \lim_{T\to\infty}W_T(f)=\lim_{T\to\infty}\inf_{u\in\mathbb{H}}F_T(u)=\inf_{u\in\mathbb{H}}F_\infty(u)=:W(f)
\end{align*}
where
\begin{align*}
F_\infty(u)=\mathbb{E}\left[\varPhi_\infty(\mathbb{Y},u)+\Vert \nabla Z_\infty(u)\Vert_{L^4}^4+\frac{1}{2}\Vert u \Vert_\mathcal{H}^2\right].
\end{align*}
\end{thm}
\begin{proof}
By Theorems \ref{minimizer}, \ref{EC}, and \ref{thm:GC} we obtain
\begin{align*}
    \lim_{T\to\infty}\inf_{\mu\in\bar{\mathcal{X}}}\tilde{F}_T(\mu)=\min_{\mu\in\bar{\mathcal{X}}}\tilde{F}_\infty(\mu).
\end{align*}
Moreover, by Theorem \ref{inf} we have for $T\in[0,\infty]$
\begin{align*}
    \inf_{u\in\mathbb{H}}F_T(u)=\inf_{\mu\in\mathcal{X}}\tilde{F}_T(\mu)=\inf_{\mu\in\bar{\mathcal{X}}}\tilde{F}_T(\mu).
\end{align*}
 Combining these equalities completes the proof.
\end{proof}
This gives us an explicit formula for the limit of the free energies and therefore an explicit description of the limit measure $\nu$.

\section{Aviles-Giga}\label{sec:AvilesGiga}
We start by considering the Hamiltonian
\begin{align}
    H(\varphi)=\int_{\mathbb{T}^2}|\nabla\varphi|^4-|\nabla\varphi|^2+|\Delta\varphi|^2\,dx.
\end{align}
Similarly as in \cite{barashkov2022variational}, we write 
\begin{align*}
    H(\varphi)=\int_{\mathbb{T}^2}|\nabla\varphi|^4-2|\nabla\varphi|^2+|\nabla\varphi|^2+|\Delta\varphi|^2\,dx,
\end{align*}
and therefore
\begin{align*}
    F_T(v)\coloneq\mathbb{E}\bigg[ f(Y_T+Z_T(u))+\int_{\mathbb{T}^d}(|\nabla Y_T+\nabla Z_T(u)|^4-2|\nabla Y_T+\nabla Z_T(u)|^2\\-\mathcal{O}_T(|\nabla Y_T+\nabla Z_T(u)|^2))\ dx+\frac{1}{2}\int_0^\infty\Vert v_s\Vert^2_{L^2}\ ds\bigg].
\end{align*}
which leads to a renormalization with the same Gaussian process $Y$ (see \eqref{eq:Y}).
\begin{thm}
    
Choosing 
\begin{align*}
     a_T=&6\langle\partial_1Y,\partial_1Y\rangle_T(\partial_1Y_T)^2+6\langle\partial_2Y,\partial_2Y\rangle_T(\partial_2Y_T)^2+2\langle\partial_2Y,\partial_2Y\rangle_T(\partial_1Y_T)^2\\
    &+2\langle\partial_1Y,\partial_1Y\rangle_T(\partial_2Y_T)^2+4\langle\partial_1Y,\partial_2Y\rangle_T^2-3\langle\partial_1Y,\partial_1Y\rangle_T^2-3\langle\partial_2Y,\partial_2Y\rangle_T^2\\
    &-2\langle\partial_1Y,\partial_1Y\rangle_T\langle\partial_2Y,\partial_2Y\rangle_T\\
    &+8 \int_0^T\int_0^s \partial_1 Y_u\ d(\partial_2 Y_u) \ d\langle\partial_1 Y,\partial_2Y\rangle_s +8 \int_0^T\int_0^s \partial_2 Y_u\ d(\partial_1 Y_u) \ d\langle\partial_1 Y,\partial_2Y\rangle_s\\
    &+8 \int_0^T\int_0^s \partial_1 Y_u\ d\langle\partial_2 Y,\partial_2 Y\rangle_u\ d(\partial_1Y_s)+8 \int_0^T\int_0^s \langle\partial_2 Y,\partial_2 Y\rangle_u\ d(\partial_1 Y_u)\ d(\partial_1Y_s)\\
    &+8 \int_0^T\int_0^s \partial_2 Y_u\ d\langle\partial_1 Y,\partial_1 Y\rangle_u\ d(\partial_2Y_s)+8 \int_0^T\int_0^s \langle\partial_1 Y,\partial_1 Y\rangle_u\ d(\partial_2 Y_u)\ d(\partial_2Y_s),\\
     b_T=&\begin{pmatrix}
        3\langle\partial_1Y,\partial_1Y\rangle_T\partial_1Y_T+2\langle\partial_1Y,\partial_2Y\rangle_T\partial_2Y_T+\langle\partial_2Y,\partial_2Y\rangle_T\partial_1Y_T\\ 
        3\langle\partial_2Y,\partial_2Y\rangle_T\partial_2Y_T+2\langle\partial_1Y,\partial_2Y\rangle_T\partial_1Y_T+\langle\partial_1Y,\partial_1Y\rangle_T\partial_2Y_T
    \end{pmatrix},
    \end{align*}
    \begin{align*}
     &c_T:\R^2\to\R,\quad c_T((x_1,x_2))=&&6\langle\partial_1Y,\partial_1Y\rangle_Tx_1^2+6\langle\partial_2Y,\partial_2Y\rangle_Tx_2^2+2\langle\partial_1Y,\partial_1Y\rangle_Tx_2^2\\
     & &&+2\langle\partial_2Y,\partial_2Y\rangle_Tx_1^2+8\langle\partial_1Y,\partial_2Y\rangle_Tx_1x_2,
     \end{align*}
     and
     \begin{align*}
         d_T=\langle\partial_1Y,\partial_1Y\rangle_T+\langle\partial_2Y,\partial_2Y\rangle_T
     \end{align*}

yields
\begin{align*}
   & \int|\nabla Y_T+\nabla Z_T(u)|^4-2|\nabla Y_T+\nabla Z_T(u)|^2-c_T(\nabla Z_T(u))-4b_T(\nabla Z_T(u))-a_T+2d_T\\
    &=\int[\![\nabla Y_T^4]\!]
    +4[\![\nabla Y_T^3]\!]\cdot\nabla Z_T(u)+4\nabla Y_T\cdot(\nabla Z_T(u))^3
    +(\nabla Z_T(u))^4\\&
    +\sum_{i,j,k,l\in\{1,2\}}2(\delta_{ik}\delta_{jl}+\delta_{il}\delta_{jk}+\delta_{ij}\delta_{kl})[\![\partial_iY_T\partial_j Y_T]\!]\partial_k Z_T(u) \partial_l Z_T(u)\\&-2[\![\nabla Y_T^2]\!]-4(\nabla Y_T)(\nabla Z_T(u))-2(\nabla Z_T(u))^2.
\end{align*}
If we define
\begin{align*}
    &\Phi_T(u)\coloneq\int4\nabla Y_T\cdot(\nabla Z_T(u))^3 +4[\![\nabla Y_T^3]\!]\cdot\nabla Z_T(u)\nonumber\\
    &+\sum_{i,j,k,l\in\{1,2\}}2(\delta_{ik}\delta_{jl}+\delta_{il}\delta_{jk}+\delta_{ij}\delta_{kl})[\![\partial_iY_T\partial_j Y_T]\!]\partial_k Z_T(u) \partial_l Z_T(u)\\
    &-4(\nabla Y_T)(\nabla Z_T(u))-2(\nabla Z_T(u))^2
\end{align*}
we obtain the representation
\begin{align*}
    F_T(u)=\mathbb{E}\left[ f(Y_T+Z_T(u))+\Phi_T(u)+\Vert\nabla Z_T(u)\Vert_{L^4}^4+\frac{1}{2}\int_0^\infty\Vert u_s\Vert^2_{L^2}\ ds\right].
\end{align*}
\end{thm}
Since 
\begin{align*}
    |\int (\nabla Y_T)(\nabla Z_T(u))|\leq \Vert \nabla Y_T\Vert_{H^{-\varepsilon}}\Vert \nabla Z_T(u)\Vert_{H^{\varepsilon}}&\leq \frac{1}{2}\Vert \nabla Y_T\Vert_{H^{-\varepsilon}}^2+\frac{1}{2}\Vert \nabla Z(u)\Vert_{H^{\varepsilon}}^2\\&\leq\frac{1}{2}\Vert \nabla Y_T\Vert_{H^{-\varepsilon}}^2+\frac{1}{2}\Vert u\Vert_{\mathcal{H}}^2\\
   \text{and\ } |\int(\nabla Z_T(u))^2|&\leq \Vert u\Vert^2_\mathcal{H}
\end{align*}
the additional terms in $\Phi_T$ allow for estimates as in Lemma \ref{lem:lubounds}. Therefore, the $\Gamma$-convergence follows as in the previous section.

\section{Extensions}\label{sec:extension}
In this section, we consider several generalizations to related Hamiltonians. We sketch how one extends the theory to these Hamiltonians and perform only the crucial computations.

\subsection{More general powers}

We consider another variation in the Hamiltonian, i.e. 
\begin{align*}
     H(\varphi)=\int_{\mathbb{T}^2}|\nabla\varphi|^p+|\nabla\varphi|^2+|\Delta\varphi|^2\,dx.
\end{align*}
for an arbitrary $p\in\N$ with $p\geq3$.\\
In two dimensions we can use Wick-renormalization as in the case $p=4$ by subtracting the drift part of the powers (see Theorem \ref{thm:renormalization}). We do not provide the explicit renormalization constants here. \\
We decompose $F_T$ as in \eqref{eq:FT_Phi} collecting all mixed terms in $\Phi_T$ and replacing $\Vert \nabla Z_T\Vert_{L^4}^4$ by $\Vert \nabla Z_T\Vert_{L^p}^p$. To obtain $\Gamma$-convergence following the steps from above we can control $\Phi_T$ similarly as in Lemma \ref{lem:lubounds}. 
\begin{lem}\label{lem:generalbounds}
    Let $p\in\N$ with $p\geq 3$ and $q\in\N$ with $q\leq p-1$. For any $\delta>0$ there exist $C<\infty$, independent of $T$, $\max\{1,\frac{p}{q}2(\frac{q}{p}-\frac{2}{p}+2)^{-1}\}<\beta<\frac{p}{q}$ and $0<\varepsilon\leq \frac{(p-2)\beta}{2(p-q\beta)}$  such that for $\alpha\in\{1,2\}^{p-k}$ and $\gamma\in\{1,2\}^q$
     \begin{align*}
        &|(\int_0^T...\int_0^{s_{k-1}} d(\partial_{\alpha_1}Y_{s_1})...d(\partial_{\alpha_k}Y_{s_{p-q}}))(\partial_{\gamma_1} Z_T(u)...\partial_{\gamma_q}Z_T(u))|\\
        &\lesssim \Vert\int_0^T...\int_0^{s_{k-1}} d(\partial_{\alpha_1}Y_{s_1})...d(\partial_{\alpha_k}Y_{s_{p-q}})\Vert_{W^{-\varepsilon,\frac{\beta}{\beta-1}}}\Vert  \nabla Z_T(u)\Vert_{W^{\varepsilon,\frac{p\beta}{p-(q-1)\beta}}}\Vert \nabla Z_T(u)\Vert_{L^p}^{q-1}
    \end{align*}
    and
    \begin{align*}
       &\mathbb{E}\left[ |(\int_0^T...\int_0^{s_{k-1}} d(\partial_{\alpha_1}Y_{s_1})...d(\partial_{\alpha_k}Y_{s_{p-q}}))(\partial_{\gamma_1} Z_T(u)...\partial_{\gamma_q}Z_T(u))|\right]\\
       &\leq C+\delta\mathbb{E}\left[\Vert \nabla Z_T(u)\Vert_{L^p}^p+\frac{1}{2}\Vert u\Vert_\mathcal{H}^2\right].
    \end{align*}
\end{lem}
\begin{proof}
  In the first step we want to prove that for all $q\in\N$ the following statement $P(q)$ holds: For all $p>q$ and all $\beta$ with $1<\beta<\frac{p}{q}$, we have for all $\gamma\in\{1,2\}^q$
  \begin{align}\label{eq:leibnitz}
      \Vert (1+D^2)^\frac{\varepsilon}{2} \partial_{\gamma_1} Z_T(u)...\partial_{\gamma_q}Z_T(u)\Vert_{L^\beta}\lesssim \Vert (1+D^2)^\frac{\varepsilon}{2} (\nabla Z_T(u))\Vert_{L^\frac{p\beta}{p-(q-1)\beta}}\Vert \nabla Z_T(u)\Vert_{L^p}^{q-1}.
  \end{align}
  For $q=1$ the statement $P(1)$  is true since for $\gamma\in\{1,2\}$
  \begin{align*}
      \Vert(1+D^2)^\frac{\varepsilon}{2} \partial_{\gamma} Z_T(u)\Vert_{L^\beta}&=\Vert (1+D^2)^\frac{\varepsilon}{2} \partial_{\gamma} Z_T(u)\Vert_{L^\frac{p\beta}{p-(1-1)\beta}}\Vert \partial_{\gamma}Z_T(u)\Vert_{L^p}^{1-1}\\
      &\leq \Vert (1+D^2)^\frac{\varepsilon}{2} (\nabla Z_T(u))\Vert_{L^\frac{p\beta}{p-(1-1)\beta}}\Vert (\nabla Z_T(u))\Vert_{L^p}^{1-1}.
  \end{align*}
  Next, we show the implication $P(q)\implies P(q+1)$. \\
  Let $1<\alpha<\frac{p}{q+1}$ and $\gamma\in\{1,2\}^{q+1}$ be arbitrary. We apply Lemma \ref{lem:leibnitz} with the choices $p_0=\alpha$, $p_1=\frac{p\alpha}{p-q\alpha}$, $p_2=\frac{p}{q}$, $p_1'=\frac{p\alpha}{p-\alpha}$, $p_2'=p$, $f=\partial_{\gamma_{q+1}}Z_T(u)$ and $g=\partial_{\gamma_1} Z_T(u)...\partial_{\gamma_{q}}Z_T(u)$ and Hölders inequality to obtain
  \begin{align*}
      &\Vert (1+D^2)^\frac{\varepsilon}{2} \partial_{\gamma_1} Z_T(u)...\partial_{\gamma_{q+1}}Z_T(u)\Vert_{L^\alpha}\\&\lesssim \Vert (1+D^2)^\frac{\varepsilon}{2} \partial_{\gamma_{q+1}}Z_T(u)\Vert_{L^\frac{p\alpha}{p-q\alpha}}\Vert \partial_{\gamma_1} Z_T(u)...\partial_{\gamma_{q+1}}Z_T(u)\Vert_{L^\frac{p}{q}}\\&
      +\Vert (1+D^2)^\frac{\varepsilon}{2} \partial_{\gamma_1} Z_T(u)...\partial_{\gamma_{q}}Z_T(u)\Vert_{L^\frac{p\alpha}{p-\alpha}}\Vert \partial_{\gamma_{q+1}}Z_T(u)\Vert_{L^p}\\
      &\lesssim \Vert (1+D^2)^\frac{\varepsilon}{2} (\nabla Z_T(u))\Vert_{L^\frac{p\alpha}{p-q\alpha}}\Vert \nabla Z_T(u)\Vert_{L^p}^q\\&
      +\Vert (1+D^2)^\frac{\varepsilon}{2} \partial_{\gamma_1} Z_T(u)...\partial_{\gamma_{q}}Z_T(u)\Vert_{L^\frac{p\alpha}{p-\alpha}}\Vert \nabla Z_T(u)\Vert_{L^p}
     .
  \end{align*}
Set $\beta=\frac{p\alpha}{p-\alpha}$. Since
\begin{align*}
  1<  \frac{p\alpha}{p-\alpha}=\frac{p}{\frac{p}{\alpha}-1}\leq \frac{p}{q-1}
\end{align*}we can apply $P(q)$  to the last term, which yields
\begin{align*}
    \Vert (1+D^2)^\frac{\varepsilon}{2} \partial_{\gamma_1} Z_T(u)...\partial_{\gamma_{q}} Z_T(u)\Vert_{L^\frac{p\alpha}{p-\alpha}}\lesssim& \Vert (1+D^2)^\frac{\varepsilon}{2} (\nabla Z_T(u))\Vert_{L^\frac{p\beta}{p-(q-1)\beta}}\Vert \nabla Z_T(u)\Vert_{L^p}^{q-1}\\
    =&\Vert (1+D^2)^\frac{\varepsilon}{2} (\nabla Z_T(u))\Vert_{L^\frac{p\alpha}{p-q\alpha}}\Vert \nabla Z_T(u)\Vert_{L^p}^{q-1}
\end{align*}
and hence
\begin{align*}
    \Vert (1+D^2)^\frac{\varepsilon}{2} \partial_{\gamma_1} Z_T(u)...\partial_{\gamma_{q+1}}Z_T(u)\Vert_{L^\alpha}\lesssim \Vert (1+D^2)^\frac{\varepsilon}{2} (\nabla Z_T(u))\Vert_{L^\frac{p\alpha}{p-q\alpha}}\Vert \nabla Z_T(u)\Vert_{L^p}^{q}.
\end{align*}
This completes the first step.\\\\
In a second step we now focus on the case $\max\{1,\frac{p}{q}2(\frac{q}{p}-\frac{2}{p}+2)^{-1}\}<\beta<\frac{p}{q}$, and further estimate \eqref{eq:leibnitz} by applying Lemma \ref{lem:interpolation}  to the first term on the right-hand side with the parameters $s=\frac{(p-2)\beta}{2(p-q\beta)}$, $p_0=\frac{p\beta}{p-(q-1)\beta}$, $s_1=1$, $s_2=0$, $p_1=2$, $p_2=p$ and $\theta=\frac{2(p-q\beta)}{(p-2)\beta}$. Note, that this choice of parameters is valid by the bounds for $\beta$. We obtain for $\gamma\in\{1,2\}^q$
\begin{align*}
    \Vert (1+D^2)^\frac{\varepsilon}{2} \partial_{\gamma_1} Z_T(u)...\partial_{\gamma_q}Z_T(u)\Vert_{L^\beta}&\lesssim \Vert (1+D^2)^\frac{\varepsilon}{2} (\nabla Z_T(u))\Vert_{L^\frac{p\beta}{p-(q-1)\beta}}\Vert \nabla Z_T(u)\Vert_{L^p}^{q-1}\\
    &\lesssim\Vert (1+D^2)^\frac{(p-2)\beta}{4(p-q\beta)} (\nabla Z_T(u))\Vert_{L^\frac{p\beta}{p-(q-1)\beta}}\Vert \nabla Z_T(u)\Vert_{L^p}^{q-1}\\
    &\lesssim\Vert \nabla Z_T(u)\Vert_{H^1}^{\frac{2(p-q\beta)}{(p-2)\beta}} \Vert \nabla Z_T(u)\Vert_{L^p}^{q-\frac{2(p-q\beta)}{(p-2)\beta}}.
\end{align*}
Finally, this estimate, Young's inequality and Lemma \ref{lem:H^1boundZ} yield
\begin{align*}
    &|( \int_0^T...\int_0^{s_{k-1}} d(\partial_{\alpha_1}Y_{s_1})...d(\partial_{\alpha_k}Y_{s_{p-q}}))(\partial_{\gamma_1} Z_T(u)...\partial_{\gamma_q}Z_T(u))|\\&\lesssim \Vert (\int_0^T...\int_0^{s_{k-1}} d(\partial_{\alpha_1}Y_{s_1})...d(\partial_{\alpha_k}Y_{s_{p-q}}))\Vert_{W^{-\varepsilon,\frac{\beta}{\beta-1}}}\Vert\partial_{\gamma_1} Z_T(u)...\partial_{\gamma_q}Z_T(u)\Vert_{W^{\varepsilon,\beta}}\\
    &\lesssim \Vert ( \int_0^T...\int_0^{s_{k-1}} d(\partial_{\alpha_1}Y_{s_1})...d(\partial_{\alpha_k}Y_{s_{p-q}}))\Vert_{W^{-\varepsilon,\frac{\beta}{\beta-1}}}\Vert \nabla Z_T(u)\Vert_{H^1}^{\frac{2(p-q\beta)}{(p-2)\beta}} \Vert \nabla Z_T(u)\Vert_{L^p}^{q-\frac{2(p-q\beta)}{(p-2)\beta}}\\
    &\lesssim \Vert (\int_0^T...\int_0^{s_{k-1}} d(\partial_{\alpha_1}Y_{s_1})...d(\partial_{\alpha_k}Y_{s_{p-q}}))\Vert_{W^{-\varepsilon,\frac{\beta}{\beta-1}}}
        \Vert u\Vert_\mathcal{H}^{\frac{2(p-q\beta)}{(p-2)\beta}}\Vert \nabla Z_T(u)\Vert_{L^p}^{q-\frac{2(p-q\beta)}{(p-2)\beta}}\\
        &\leq C(\delta)  \Vert ( \int_0^T...\int_0^{s_{k-1}} d(\partial_{\alpha_1}Y_{s_1})...d(\partial_{\alpha_k}Y_{s_{p-q}}))\Vert_{W^{-\varepsilon,\frac{\beta}{\beta-1}}}^\frac{\beta}{\beta-1}+\delta\left(\Vert \nabla Z_T(u)\Vert_{L^p}^p+\frac{1}{2}\Vert u\Vert_\mathcal{H}^2\right).
\end{align*}
Taking the expectation together with Lemma \ref{WickRegularity} completes the proof.
\end{proof}
\begin{rmk}
    The martingales defined in Lemma \ref{Wick-Powers} are iterated integrals and therefore are of the form appearing in the first factor of the estimate.
\end{rmk}

For the convergence of $\Vert \nabla Z_T\Vert_{L^p}^p$ we need the following compactness result for $\nabla Z$.

\begin{lem}\label{compactZ}
    For $1\leq p<\infty$ and $\tilde{\varepsilon}<\frac{2}{p}$, the map $\nabla Z:\mathcal{H}\to C([0,\infty];W^{\tilde{\varepsilon},p}(\mathbb{T}^2;\R^2))$ is compact.
\end{lem}
\begin{proof} To show that the image of a bounded set in $\mathcal{H}$ is relatively compact in \\$C([0,\infty];W^{\tilde{\varepsilon},p}(\mathbb{T}^2;\R^2))$ we use the Arzelà-Ascoli theorem, which says that a set $F\in C(X,Y)$ is relatively compact if and only if $F$ is equicontinuous and pointwise relatively compact. Let $\tilde{\varepsilon}<\varepsilon<\frac{2}{p}$. With the Sobolev embedding $W^{1-\kappa,2}\hookrightarrow W^{\varepsilon,p}$  for $0<\kappa<\frac{2}{p}-\varepsilon$ and the Mihlin-Hörmander multiplier theorem (using the explicit form of $\sigma_t$ in the definition of $J_t$, see \eqref{eq:Jt}), we get
\begin{align*}
 \Vert \nabla Z_{t_1}(u)-\nabla Z_{t_2}(u)\Vert_{W^{\varepsilon,p}}=&
    \Vert \int_{t_1}^{t_2}\nabla J_su_s\ ds\Vert_{W^{\varepsilon, p}}\leq  \int_{t_1}^{t_2}\Vert \nabla J_su_s\Vert_{W^{\varepsilon, p}}\ ds\\
\lesssim&\int_{t_1}^{t_2}\Vert \nabla J_su_s\Vert_{W^{1-\kappa,2}}\ ds
\lesssim\int_{t_1}^{t_2}\Vert u_s\Vert_{L^2}\ \frac{ds}{(1+s^2)^\frac{1+2\kappa}{4}}\\
    \leq& \left(\int_{t_1}^{t_2}\Vert u_s\Vert^2_{L^2}\ ds\right)^\frac{1}{2}\left( \int_{t_1}^{t_2} \frac{ds}{(1+s^2)^\frac{1+2\kappa}{2}}\right)^\frac{1}{2}\\
    =&\left( \int_{t_1}^{t_2} \frac{ds}{(1+s^2)^\frac{1+2\kappa}{2}}\right)^\frac{1}{2}\Vert u\Vert_\mathcal{H}.
\end{align*}
Hence, equicontinuity follows with
\begin{align*}
    0\leq\lim_{t_1\to t_2}\int_{t_1}^{t_2} \frac{ds}{(1+s^2)^\frac{1+2\kappa}{2}}\leq |t_2-t_1|\stackrel{t_1\to t_2}{\to}0.
\end{align*}
Furthermore, since 
\begin{align*}
    \int_{0}^{\infty} \frac{ds}{(1+s^2)^\frac{1+2\kappa}{2}}=:C<\infty
\end{align*}
we have
\begin{align*}
    \Vert Z_{t_1}(u)\Vert_{W^{\varepsilon,p}}\leq C \Vert u\Vert_\mathcal{H}.
\end{align*}
By Rellich-Kondrachov theorem, the embedding $W^{\varepsilon,p}\stackrel{c}{\hookrightarrow}W^{\tilde{\varepsilon},p}$ is compact, and hence, we obtain the required 
pointwise relative compactness of images of bounded sets under $\nabla Z$.
\end{proof}

\subsection{Three Dimensions}
While so far we focussed on the two-dimensional case, we point out that the strategy can also be extended to three dimensions. We comment on  the differences here, and note that in \cite{Barashkov_2020} the computations for the three-dimensional setting are presented in great detail for the $\Phi_4^3$-model.\\
In three dimensions Wick-renormalization is not sufficient since the Wick-powers are less regular, see explicit computations in Appendix \ref{Wick}. So we need an additional renormalization. Similarly to the processes in Lemma \ref{Wick-Powers} the notation $[\![\nabla Y_T^i]\!]$ describes the martingale part of the $i$-th power of $\nabla Y$. In the following, we use the notation $\circ$ and $\prec$ as introduced in Definition \ref{def:paraproducts}, and $Z^b_T(u)$ is a regularized version of $Z_T(u)$, for details see \cite[Chapter 4]{Barashkov_2020}.

\begin{prop}\label{prop:renormalization3D}
    
Choosing 
\begin{align*}
     &\gamma_T:\R^3\times\R^3\to\R^3,\quad \gamma_T(z)(p)=&&\sum_{i,j,k,l\in\{1,2,3\}}(576\sum_{q_1,q_2\in\Z^3\setminus\{0\}}\\& &&(q_1^i q_2^j)\int_0^T\int_0^T\frac{\sigma^2_{r_1}(q_1)}{(|q_1|^2+|q_1|^4) } \frac{\sigma^2_{r_2}(q_2)}{(|q_2|^2+|q_2|^4) }\\
     & &&(-q^k_1-q^k_2)(z(q_1\cdot q_2)+q_1(z\cdot q_2)+q_2(z\cdot q_1))\\
     & &&\int_{r_1\vee r_2}\frac{\sigma^2_{u}(-q_1-q_2)}{(|-q_1-q_2|^2+|-q_1-q_2|^4) }\ du\ dr_1\ dr_2)p^l,\\
     \end{align*}\vspace{-1.5cm}
     \begin{align*}
     \delta_T=&\mathbb{E}\bigg[-32\int_0^T\int(\nabla\cdot J_t  [\![\nabla Y_t^3]\!])^3\ dt+96\int  [\![\nabla Y_T^2]\!](\nabla\int_0^T J_t(\nabla\cdot J_t  [\![\nabla Y_t^3]\!]) ds)^2 \\
     &-8\gamma_T(\nabla Y_T)(\nabla\int_0^T J_t(\nabla\cdot J_t  [\![\nabla Y_t^3]\!])ds)+256\nabla Y_T\cdot (\nabla\int_0^T J_t(\nabla\cdot J_t  [\![\nabla Y_t^3]\!])ds)^3 \bigg]
\end{align*}
yields
\begin{align}\label{eq:FT_3d}
   F_T(u)
    =&\mathbb{E}\bigg[f(Y_T+Z_T)+ 4\int [\![\nabla Y_T^3]\!]\cdot\nabla Z_T(u)+4 \int \nabla Y_T\cdot(\nabla Z_T(u))^3\nonumber\\&+\sum_{i,j,k,l\in\{1,2,3\}}2(\delta_{ik}\delta_{jl}+\delta_{il}\delta_{jk}+\delta_{ij}\delta_{kl})[\![\partial_iY_T\partial_j Y_T]\!]\partial_k Z_T(u) \partial_l Z_T\nonumber\\&-2 \int \gamma_T( \nabla Y_T)(\nabla Z_T(u))+\int  \gamma_T(\nabla Z_T(u))(\nabla Z_T(u))+\delta_T\bigg]\nonumber\\&+\mathbb{E}\bigg[ \int  (\nabla Z_T(u))^4+\frac{1}{2}\Vert u\Vert^2_\mathcal{H}\bigg]\nonumber\\
    =&\mathbb{E}\left[ f(Y_T+Z_T(u))+\Phi_T(u)+\Vert\nabla Z_T(u)\Vert_{L^4}^4+\frac{1}{2}\Vert l^T(u)\Vert_\mathcal{H}^2\right],
\end{align}
where
\begin{align}\label{eq:PhiT3D}
\Phi_T(u)&\coloneq \sum_{i=1}^6\mathcal{Y}^{(i)}_T
\end{align}\vspace{-0.1pt}
    with
\begin{align*}
    \mathcal{Y}^{(1)}_T&\coloneq \frac{1}{2}\int_\Lambda (12\sum_{i,j,k,l\in\{1,2\}}2(\delta_{ik}\delta_{jl}+\delta_{il}\delta_{jk}+\delta_{ij}\delta_{kl})(-([\![\partial_iY_T\partial_j Y_T]\!]\succ\partial_k Z_T(w)) \partial_l Z_T(w)\\&-([\![\partial_iY_T\partial_j Y_T]\!]\circ\partial_k Z_T(w)) \partial_l Z_T(w)+([\![\partial_iY_T\partial_j Y_T]\!]\prec\partial_k Z_T(w))\partial_l Z_T(w))\\&+ 4([\![\partial_iY_T\partial_j Y_T]\!]\prec(\partial_k\int_0^TJ_t(\nabla\cdot J_t [\![\nabla Y_T^3]\!]) dt)) \partial_l Z_T(w)),\\
\mathcal{Y}^{(2)}_T&\coloneq \int_\Lambda 12\sum_{i,j,k,l\in\{1,2\}}2(\delta_{ik}\delta_{jl}+\delta_{il}\delta_{jk}+\delta_{ij}\delta_{kl})(([\![\partial_iY_T\partial_j Y_T]\!]\succ\partial_k (Z_T(u)-Z_T^b(u))) \partial_l Z_T(w)),\\
\mathcal{Y}^{(3)}_T&\coloneq \int_0^T \int_\Lambda 12\sum_{i,j,k,l\in\{1,2\}}2(\delta_{ik}\delta_{jl}+\delta_{il}\delta_{jk}+\delta_{ij}\delta_{kl})(([\![\partial_iY_t\partial_j Y_T]\!]\succ\partial_k Z_t^b(u)) \partial_l Z_t(w)),\\
\mathcal{Y}^{(4)}_T&\coloneq 4  \int_\Lambda  \nabla Y_T\cdot(\nabla Z_T(w))^3+4 \int_\Lambda  \nabla\ Y_T \cdot(4\nabla\int_0^TJ_t(\nabla\cdot J_t [\![\nabla Y_T^3]\!]) dt) (\nabla Z_T(w))^2\\
&+ 8 \int_\Lambda ( \nabla\ Y_T \cdot\nabla Z_T(w))((4\nabla\int_0^TJ_t(\nabla\cdot J_t [\![\nabla Y_T^3]\!]) dt)\cdot \nabla Z_T(w))\\&+ 4 \int_\Lambda  (4\nabla\int_0^TJ_t(\nabla\cdot J_t [\![\nabla Y_T^3]\!]) dt)^2\nabla Y_T\cdot \nabla Z_T(w)\\
&+8 \int_\Lambda(  (4\nabla\int_0^TJ_t(\nabla\cdot J_t [\![\nabla Y_T^3]\!]) dt)\cdot\nabla Y_T)((4\nabla\int_0^TJ_t(\nabla\cdot J_t [\![\nabla Y_T^3]\!]) dt)\cdot \nabla Z_T(w)),\\
\mathcal{Y}^{(5)}_T&\coloneq - \int_\Lambda \gamma_T(\nabla Z_T^b(u))(\nabla Z_T(u)-\nabla Z_T^b(u))- \int_\Lambda \gamma_T(\nabla Z_T(u)-\nabla Z_T^b(u))(\nabla Z_T^b(u))\\&
- \int_\Lambda \gamma_T(\nabla Z_T(u)-\nabla Z_T^b(u))(\nabla Z_T(u)-\nabla Z_T^b(u))-\int_0^T \int_\Lambda\gamma_t(\nabla Z_t^b(u))(\nabla\dot{Z}_t^b(u))\ dt\\&-\int_0^T \int_\Lambda\gamma_t(\nabla \dot{Z}_t^b(u))(\nabla Z_t^b(u))\ dt,\\
\mathcal{Y}^{(6)}_T&\coloneq \int_\Lambda (48\sum_{i,j,k,l\in\{1,2\}}2(\delta_{ik}\delta_{jl}+\delta_{il}\delta_{jk}+\delta_{ij}\delta_{kl})([\![\partial_iY_T\partial_j Y_T]\!]\circ(\partial_k\int_0^TJ_t(\nabla\cdot J_t [\![\nabla Y_T^3]\!]) dt)) \partial_l Z_T(w)))\\&-\int_\Lambda 2\gamma_T(\nabla Y)(\nabla Z_T(w))\\&
-\frac{1}{2}\int_0^T \int_\Lambda \sum_{\substack{i,j,k,l, \tilde{i},\tilde{j},\\\tilde{k},\tilde{l} \in\{1,2,3\}}}(\delta_{ik}\delta_{jl}+\delta_{il}\delta_{jk}+\delta_{ij}\delta_{kl})(\delta_{\tilde{i}\tilde{k}}\delta_{\tilde{j}\tilde{l}}+\delta_{\tilde{i}\tilde{l}}\delta_{\tilde{j}\tilde{k}}+\delta_{\tilde{i}\tilde{j}}\delta_{\tilde{k}\tilde{l}})\\& \bigg( \partial_l J_t ( [\![\partial_iY_t\partial_j Y_T]\!])\circ \partial_{\tilde{l}}J_t(  [\![\partial_{\tilde{i}}Y_t\partial_{\tilde{j}} Y_T]\!]) \bigg)\partial_k Z^b_T(u)\partial_{\tilde{k}} Z_T^b(u)+\dot{\gamma}_t(\nabla Z_t^b(u))(\nabla Z_t^b(u))\ dt
\\&
-\frac{1}{2}\int_0^T\int_\Lambda  (J_t(12\sum_{i,j,k,l\in\{1,2\}}2(\delta_{ik}\delta_{jl}+\delta_{il}\delta_{jk}+\delta_{ij}\delta_{kl})\partial_l (([\![\partial_iY_t\partial_j Y_T]\!]\succ\partial_k Z_t^b(u)))^2\\&-\sum_{\substack{i,j,k,l, \tilde{i},\tilde{j},\\\tilde{k},\tilde{l} \in\{1,2,3\}}}(\delta_{ik}\delta_{jl}+\delta_{il}\delta_{jk}+\delta_{ij}\delta_{kl})(\delta_{\tilde{i}\tilde{k}}\delta_{\tilde{j}\tilde{l}}+\delta_{\tilde{i}\tilde{l}}\delta_{\tilde{j}\tilde{k}}+\delta_{\tilde{i}\tilde{j}}\delta_{\tilde{k}\tilde{l}})\\& \bigg( \partial_l J_t ( [\![\partial_iY_t\partial_j Y_T]\!])\circ \partial_{\tilde{l}}J_t(  [\![\partial_{\tilde{i}}Y_t\partial_{\tilde{j}} Y_T]\!]) \bigg)\partial_k Z^b_T(u)\partial_{\tilde{k}} Z_T^b(u)\ dt,
    \end{align*}
    and 
     \begin{align*}
        l^T_t(u)&\coloneq u_t-4\mathbbm{1}_{t\leq T}\nabla\cdot  J_t[\![\nabla Y_T^3]\!]\\&-\mathbbm{1}_{t\leq T} J_t(12\sum_{i,j,k,l\in\{1,2\}}2(\delta_{ik}\delta_{jl}+\delta_{il}\delta_{jk}+\delta_{ij}\delta_{kl})\partial_l (([\![\partial_iY_t\partial_j Y_T]\!]\succ\partial_k Z_t^b(u)) ,\\
        w^T_t(u)&\coloneq u_t-4\mathbbm{1}_{t\leq T}\nabla\cdot  J_t[\![\nabla Y_T^3]\!].
    \end{align*}
   
\end{prop}

\begin{proof} The decomposition works analogously to \cite[Lemma 5]{Barashkov_2020} for the $\Phi^4_3$ model.  For the readers' convenience, we briefly sketch the main ideas. The idea is to use Itô's formula for the products $[\![\nabla Y_T^2]\!](\nabla Z_T(u))^2$ and $[\![\nabla Y_T^3]\!]\cdot\nabla Z_T(u)$ and extracting the divergent parts, i.e. the parts where we have no control over the high frequencies of the singular term $\nabla Y$, by using paraproducts, see Definition \ref{def:paraproducts} and Remark \ref{rmk:paraproducts}.  \\
As in the two dimensional case, see Theorem \ref{thm:renormalization}, the term where $\nabla Y$ has quadratic influence appears as a sum of several terms with mixed partial derivatives. We renormalize each term separately as in the scalar case for the $\Phi^4_3$ model. \\
  In addition, due to the vectorial gradient structure one has to integrate by parts to match the regularity and dimensions. This leads to derivatives that are sometimes of the form of a divergence which would not appear in the scalar case.\\
To illustrate this we perform exemplarily the computation for $[\![\nabla Y_T^3]\!]\cdot\nabla Z_T(u)$.\\
Note that $\nabla Z(u)$ is of bounded variation. Itô's formula yields
\begin{align*}
    \int_{\mathbb{T}^3} [\![\nabla Y_T^3]\!]\cdot\nabla Z_T(u)=\int_0^T\int_{\mathbb{T}^3} [\![\nabla Y_t^3]\!]\cdot\nabla \dot{Z}_t(u)\ dt+ \int_0^T\int_{\mathbb{T}^3} \nabla Z_t(u)\cdot d([\![\nabla Y_t^3]\!]).
\end{align*}
The second term on the right hand-side is a martingale and vanishes in expectation, so we continue with the first term. We have
\begin{align*}
    &4\int_0^T\int_{\mathbb{T}^3} [\![\nabla Y_t^3]\!]\cdot\nabla \dot{Z}_t(u)\ dt
    =4\int_0^T\int_{\mathbb{T}^3} [\![\nabla Y_t^3]\!]\cdot\nabla J_tu\ dt\\
    &=4\int_0^T\int_{\mathbb{T}^3} [\![\nabla Y_t^3]\!]\cdot\nabla J_tu\ dt\\
    &+\frac{16}{2}\int_0^T\int_{\mathbb{T}^3} (\nabla\cdot J_t[\![\nabla Y_t^3]\!])^2\ dt-\frac{16}{2}\int_0^T\int_{\mathbb{T}^3} (\nabla\cdot J_t[\![\nabla Y_t^3]\!])^2\ dt\\
    &+\frac{1}{2}\int_0^\infty\int_{\mathbb{T}^3} (u_t)^2\ dt-\frac{1}{2}\int_0^\infty\int_{\mathbb{T}^3} (u_t)^2\ dt\\
    &=\frac{1}{2}\int_0^\infty\int_{\mathbb{T}^3} (u_t-4\mathbbm{1}_{t\leq T}\nabla\cdot J_t[\![\nabla Y_t^3]\!])^2\ dt\\&-\frac{16}{2}\int_0^T\int_{\mathbb{T}^3} (\nabla\cdot J_t[\![\nabla Y_t^3]\!])^2\ dt-\frac{1}{2}\int_0^\infty\int_{\mathbb{T}^3} (u_t)^2\ dt.
\end{align*}
Consider the last expression. The last term cancels with $\frac{1}{2}\Vert u\Vert_\mathcal{H}^2$, and the second to last term does not depend on $u$ anymore and hence has no influence on the infimizing process. Finally, the first term is part of the perturbed drift $l_T(u)$ in \eqref{eq:FT_3d}.
\end{proof}

In \eqref{eq:FT_3d}, the term $l^T(u)$ is now a perturbation  of $u$ containing terms depending on $Y_T$, and $\Phi_T$ collects terms with bounds depending on $\Vert\nabla Z_T(u)\Vert_{L^4}^4$, $\frac{1}{2}\Vert l_T(u)\Vert^2_\mathcal{H}$ and hence has a good convergence behavior in the sense of $\Gamma$-convergence, see \cite[Lemma 8]{Barashkov_2020}.  
\\ 
 
To prove the equicoercivity and $\Gamma$-convergence one proceeds along the lines of \cite[Chapter 6]{Barashkov_2020}.

\quad\\
\textbf{Acknowledgement}\\
We acknowledge funding by the Deutsche Forschungsgemeinschaft (DFG, German Research Foundation) – CRC/TRR 388 "Rough Analysis, Stochastic Dynamics and Related Fields“ – Project ID 516748464",
and CT acknowledges funding by the Deutsche Forschungsgemeinschaft (DFG, German Research Foundation) under 
Germany´s Excellence Strategy – The Berlin Mathematics
Research Center MATH+ (EXC-2046/1, EXC-2046/2, project ID: 390685689). The authors thank  Felix Otto and Nicolas Perkowski for helpful discussions and insights.
\newpage

\bibliography{bibliography}
\bibliographystyle{plain}

\section*{Appendix}
\appendix

\section{Besov Spaces}\label{sec:besov}
In order to characterize the support of the Gaussian measure we introduce Besov spaces.
\begin{defi}[Besov Spaces]\label{def:besov}
Let $\rho,\chi\in C^\infty(\R^d)$ such that
\begin{itemize}
   \item $ \supp \rho \subset B_{2R}(0)\setminus B_R(0) \text{ and } \supp \chi\subset B_R(0)$
    \item$\rho\leq 1,\chi\geq 0 \text{ and } \chi(\xi)+ \sum_{j\geq 0}\rho(2^{-j}\xi)=1 \quad \text{for all }\xi\in\R^d  $
\end{itemize}
and denote $\rho_j=\rho(2^{-j}\cdot)$ for $j\geq 0$ and $\rho_{-1}=\chi$ and $\Delta_jf\coloneq \mathcal{F}^{-1}\rho_j\mathcal{F}f$. \\
 We define for $\alpha\in\R$  
 \begin{align*}
     B_{p,q}^\alpha\coloneq \{f\in\mathcal{S}'(\mathbb{T}^d):\Vert f\Vert_{B^\alpha_{p,q}}<\infty\}
 \end{align*}
 with
 \begin{align*}
     \Vert f\Vert_{B^\alpha_{p,q}}\coloneq \Vert (2^{\alpha j}\Vert \Delta_jf\Vert_{L^p})_{j\geq-1}\Vert_{l^q}.
 \end{align*}
 For $p=q=2$ we have the fractional Sobolev spaces $H^\alpha\coloneq B_{p,q}^\alpha$.\\
 Moreover, we define
 \begin{align*}
     \mathscr{C}^\alpha\coloneq \overline{C(\mathbb{T}^d)}^{B_{\infty,\infty}^\alpha}.
 \end{align*}

\end{defi}
\begin{rmk}
    For $\alpha\in\N$ the fractional Sobolev spaces agree with the classical Sobolev spaces.
\end{rmk}
\begin{defi}\label{def:paraproducts}
    Let $f,g\in\mathcal{S}'(\mathbb{T}^d)$. We define the paraproducts and resonant product as 
    \begin{align*}
        f\succ g=g\prec f&\coloneq \sum_{j<i-1}\Delta_if\Delta_jg,\\
        f\circ g&\coloneq \sum_{|i-j|\leq 1}\Delta_if\Delta_jg.
    \end{align*}
\end{defi}
\begin{rmk}\label{rmk:paraproducts}
    The regularity of $f\succ g$ depends highly on the regularity of $f$. In particular, if $f$ is singular and $g$ is not, the paraproduct $f\succ g$ will adopt the regularity of $f$ whereas $g\succ f$ and $f\circ g$ behave better. For more details see \cite[Lemma 2.1]{GUBINELLI_IMKELLER_PERKOWSKI_2015}.
\end{rmk}
The following lemma provides some useful embeddings, see \cite{Triebel1983} for more details.
\begin{lem}\label{Besovnormestimates}
    Let $\delta>0$. We have for any $s\in\R$ and $p,q_1,q_2\in[1,\infty]$, $q_1<q_2$
    \begin{align}\label{besovembedding1}
        \Vert f\Vert_{B_{p,q_2}^s}\lesssim \Vert f\Vert_{B_{p,q_1}^s}\lesssim \Vert f\Vert_{B_{p,\infty}^{s+\delta}}\quad \forall f\in\mathcal{S}'(\mathbb{T}^d).
    \end{align}
    If we denote the fractional Sobolev spaces by $W^{s,p}$ then for any $q\in[1,\infty]$
    \begin{align}\label{besovembedding2}
        \Vert f\Vert_{B_{p,q}^s}\lesssim \Vert f\Vert_{W^{s+\delta,p}}\lesssim \Vert f\Vert_{B_{p,\infty}^{s+2\delta}}\quad \forall f\in\mathcal{S}'(\mathbb{T}^d),
    \end{align}
    and
    \begin{align}\label{besovembedding4}
        \Vert f\Vert_{W^{s,p}}\lesssim\Vert f\Vert_{B_{p,q}^{s+\delta}}\quad \forall f\in\mathcal{S}'(\mathbb{T}^d).
    \end{align}
    For all $\alpha\in\N_0^d$ and for all $s,t\in\R$, $p_1,p_2,q_1,q_2\in [1,\infty]$, with
    \begin{align*}
        p_2\geq p_1,\quad q_2\geq q_1,\quad t\leq s-d(\frac{1}{p_1}-\frac{1}{p_2})
    \end{align*}
    one has
    \begin{align}\label{besovembedding3}
        \Vert \partial^\alpha f\Vert_{B^{t-|\alpha|}_{p_2,q_2}}\lesssim\Vert f\Vert_{B^s_{p_1,q_1}} \quad \forall f\in\mathcal{S}'(\mathbb{T}^d).
    \end{align}
\end{lem}

\begin{lem}\label{lem:leibnitz}
    Let $s\geq0$ and $p_0,p_1,p_2,p_1',p_2'>1$ such that $\frac{1}{p_1}+\frac{1}{p_2}=\frac{1}{p_0}=\frac{1}{p_1'}+\frac{1}{p_2'}$. Then there exists a constant $C$ such that
    \begin{align*}
        \Vert (1+D^2)^\frac{s}{2}(fg)\Vert_{L^{p_0}}\leq C\Vert (1+D^2)^\frac{s}{2}f\Vert_{L^{p_1}}\Vert g\Vert_{L^{p_2}}+C\Vert (1+D^2)^\frac{s}{2}g\Vert_{L^{p_1'}}\Vert f\Vert_{L^{p_2'}}
    \end{align*}
    for all functions $f\in W^{s,p_1}\cap L^{p_2'}$ and $g\in W^{s,p_1'}\cap L^{p_2}$.
\end{lem}
\begin{proof}
    See \cite{1f76562f-a7be-3b2a-b3d8-37d20668fca0}.
\end{proof}
\begin{lem}\label{lem:interpolation}
    Let $0\leq\theta\leq1$, $p_0,p_1,p_2>1$ and $s,s_1,s_2\geq0$ be such that 
    \[\frac{1}{p_0}=\frac{\theta}{p_1}+\frac{1-\theta}{p_2}\quad\text{and}\quad s=\theta s_1+(1-\theta) s_2.\] 
    Then
    \begin{align*}
        \Vert f\Vert_{W^{s,p_0}}\leq\Vert f\Vert_{W^{s_1,p_1}}^\theta \Vert f\Vert_{W^{s_2,p_2}}^{1-\theta}
    \end{align*}
    for all $f\in W^{s_1,p_1}\cap W^{s_2,p_2}$.
\end{lem}
\begin{proof}
    See \cite{brezis:hal-01626613}.
    \end{proof}

\section{Gaussian Measure}\label{GaussianMeasure}

We collect some useful concepts, details can be found, e.g.,  in \cite{stroock2023gaussian} and \cite{Werner:2020rfd}.
\begin{defi}[Gaussian Random Variable]\label{definitionGRV}
 Let $(\Omega,\mathcal{F},\mathbb{P})$ be a probability space. We call a random variable $X:\Omega\to \R$ Gaussian if $X_*\mathbb{P}=\mathbb{P}\circ X^{-1}$ is a Gaussian measure on $\R$.\\
 Let $B$ be a Banach space. A $B$-valued random variable $X$ is said to be a centered Gaussian random variable if, for all $\xi\in B'$, $\langle X,\xi\rangle_{B,B'}$ is a Gaussian random variable with mean zero.
\end{defi}

\begin{defi}[Gaussian Process]
A stochastic process $(X_t)_{t\in T}$, where $T$ is some index set, not necessarily countable or finite-dimensional, is called centered Gaussian process if for any finite set $t_1<t_2<...<t_n$ the vector random variable $(X_{t_1},X_{t_2},...,X_{t_n})$ is Gaussian with mean zero. The covariance is given by $c(t_1,t_2)=\mathbb{E}[X_{t_1}X_{t_2}]$.\\
 It is common to call $X$ a random field if $T\not\subset \R$.
\end{defi}

\begin{defi}[Centered Gaussian Family]
    Given a probability space $(\Omega,\mathcal{F},\mathbb{P})$, a Gaussian family is a subspace $G$ of $L^2(\mathbb{P})$ of which each element is Gaussian. It is centered if all its elements have mean zero.
\end{defi}
\begin{defi}[Covariance Function]
Given a Gaussian family $G$, we call $c:G^2\to\R,\ c(\xi,\eta)=\int_\Omega \xi(x)\eta(x)\ \mathbb{P}(dx)$ the covariance function.
\end{defi}
%Definition of Gaussian family and covariance fuction
Now, we are able to define a Gaussian measure on an infinite-dimensional Banach space $B$. 
\begin{defi}[Gaussian Measure]\label{gaussianmeasure}
  A Borel probability measure $\mathcal{W}$ on $B$ is said to be a Gaussian measure if $\{ \langle\cdot,\xi\rangle_{B,B'}:\xi\in B'\}$  is a centered Gaussian family under $\mathcal{W}$.
\end{defi}
Note that in this case the covariance function has the following form:
\begin{align*}
    c: B'^2\to\R,\ (\xi,\eta)\mapsto \int_B \langle x,\xi\rangle_{B,B'} \langle x,\eta\rangle_{B,B'}\ \mathcal{W}(dx).
\end{align*}\\\\
In the following we construct the Gaussian measure as a distribution of a Gaussian random field. We consider the Greens's function that corresponds to $(-\Delta+\Delta^2)$ with periodic boundary conditions on $\mathbb{T}^d$ with zero mean:
\begin{align*}
    G(x,y)=\sum_{k\in\Z^d\setminus\{0\}}\frac{e^{2\pi i k\cdot(x-y)}}{|k|^2(2\pi)^2+|k|^4(2\pi)^4}.
\end{align*}
Having this we define a Gaussian process $(\Gamma(f))_{f\in L^2(\mathbb{T}^d)}$ through the covariance function 
\begin{align*}
    c(f,g)=\mathbb{E}[\Gamma(f)\Gamma(g)]=\int_{\mathbb{T}^d\times\mathbb{T}^d}G(x,y)f(x)g(y)\ dx \ dy=\langle f,(-\Delta+\Delta^2)^{-1}g\rangle_{L^2(\mathbb{T}^d)}.
\end{align*}
The existence of this field follows by the symmetry and positive-definiteness of $c$. \\
To get that the trajectories are bounded linear functionals we need to restrict $\Gamma$ to a smaller function space. \\
It is well known that there is an orthonormal basis $(\varphi_i)_{i\in\N}$ of eigenfunctions of $(-\Delta+\Delta^2)$ in $L^2(\mathbb{T}^d)$. Hereafter, we exclude the constant function. We write $\lambda_i$ for the associated positive eigenvalues. Note that the Gaussian Free Field on periodic mean zero $L^2$-functions is completely determined by $(\Gamma(\varphi_i))_{i\in\N}$ through 
\begin{align*}
    \Gamma(f)=\sum_{i\in\N}\langle f,\varphi_i\rangle_{L^2} \Gamma(\varphi_i)
\end{align*}
since $(\lambda_i^\frac{1}{2}\Gamma(\varphi_i))_{i\in\N}$ is a sequence of independent standard Gaussian random variables. This definition provides linearity of $\Gamma$.\\
It remains to show that this sum converges absolutely. First, note that by Weyl's law $\sum_{i\in\N}|\lambda_i^{-\beta}|<\infty$ if and only if $\beta>\frac{d}{4}$. Next, we define a subspace $\mathscr{H}^s$ of $L^2(\mathbb{T}^d)$ with zero mean through
\begin{align*}
    \Vert f\Vert_{\mathscr{H}^s}^2\coloneq \sum_{i\in\N}\lambda_i^s\langle f,\varphi_i\rangle_{L^2}^2<\infty.
\end{align*}
Let $s>\frac{d}{4}-1$. This space coincides with the Sobolev space $\dot{H}^s$, so we use this notation from now on. Then Cauchy-Schwarz provides for $f\in \dot{H}^s$
\begin{align*}
    \sum_{i\in\N}|\langle f,\varphi_i\rangle_{L^2} \Gamma(\varphi_i)|\leq \left(\sum_{i\in\N}\lambda_i^s\langle f,\varphi_i\rangle_{L^2}^2\right)^\frac{1}{2}\left(\sum_{i\in\N}\frac{\Gamma(\varphi_i)^2}{\lambda_i^s}\right)^\frac{1}{2}.
\end{align*}
The right-hand side is finite for Lebesgue-almost every $x\in{\mathbb{T}^d}$, because 
\begin{align*}
    \sum_{i\in\N}\frac{\mathbb{E}[\lambda_i\Gamma(\varphi_i)^2]}{|\lambda_i^{s+1}|}=\sum_{i\in\N}\frac{1}{|\lambda_i^{s+1}|}<\infty.
\end{align*}
Since we want almost every trajectory to be in the dual space of $\dot{H}^s$, we need to check continuity.
\begin{align*}
    |\Gamma(f)-\Gamma(g)|=|\Gamma(f-g)|&\leq \sum_{i\in\N}|\langle f-g,\varphi_i\rangle_{L^2} \Gamma(\varphi_i)|\\&\leq \left(\sum_{i\in\N}\lambda_i^s\langle f-g,\varphi_i\rangle_{L^2}^2\right)^\frac{1}{2}\left(\sum_{i\in\N}\frac{\Gamma(\varphi_i)^2}{\lambda_i^s}\right)^\frac{1}{2}\\
    &=\left(\sum_{i\in\N}\frac{\Gamma(\varphi_i)^2}{\lambda_i^s}\right)^\frac{1}{2}\Vert f-g\Vert_{H^{2s}}.
\end{align*}
At the beginning it was mentioned that the distribution of this field is a Gaussian measure on $(\dot{H}^{2s})'$. We write $\tilde{\Gamma}$ for our Gaussian free field, but as a function on ${\mathbb{T}^d}$ that maps to the trajectories $\{f|f:\dot{H}^{2s}\to\R\}$. We already showed in the preceding computation that for Lebesgue-almost every $x\in{\mathbb{T}^d}$ the value $\tilde{\Gamma}(x)$ is an element of $(\dot{H}^{2s})'$. The distribution is given by
\begin{align*}
    \tilde{\Gamma}_*\mathcal{L}^d:\mathscr{B}_{(H^{2s})'}\to \R.
\end{align*}
Now, it is easy to check that this is indeed a Gaussian measure in the sense of Definition \ref{gaussianmeasure}. For every $\xi\in \dot{H}^{2s}$ the random variable $\langle\cdot,\xi\rangle_{(H^{2s})',H^{2s}}$ is Gaussian with respect to $\tilde{\Gamma}_*\mathcal{L}^d$ since
\begin{align*}
     (\langle\cdot,\xi\rangle_{(H^{2s})',H^{2s}})_*\tilde{\Gamma}_*\mathcal{L}^d= \mathcal{L}^d(\tilde{\Gamma}^{-1}(\langle\cdot,\xi\rangle_{(H^{2s})',H^{2s}}^{-1})= \mathcal{L}^d(\langle\tilde{\Gamma},\xi\rangle_{(H^{2s})',H^{2s}}^{-1})= \mathcal{L}^d(\Gamma(\xi)^{-1})
\end{align*}
and $\Gamma(\xi)$ is a Gaussian random variable for every $\xi\in \dot{H}^{2s}$.

\section{Proof of Lemma 2.1}
\begin{lem}\label{lem:F*proof}
  Let $\mu$ be a probability measure on $X$. Then the functional $F:L^\infty(X)\to \R$, $f\mapsto \log\left(\int_X\exp(f(x))\ \mu(dx)\right)$ is continuous, convex and for $\gamma\in P_\mu(X)\subset (L^\infty(X))'$ the Fenchel conjugate is of the form
  \begin{align*}
      F^\ast(\gamma)=\int_X\log\left(\frac{d\gamma}{d\mu}(x)\right)\ \gamma(dx).
  \end{align*}
\end{lem}
\begin{proof}
    We start with the convexity. Let $f,g\in L^\infty(X)$ and $\lambda\in(0,1)$. Then by H\"older's inequaliy,
\begin{align*}
    F(\lambda f+(1-\lambda)g)&=\log\left(\int_X\exp(\lambda f(x)+(1-\lambda)g(x))\ \mu(dx)\right)\\
    %&=\log\left(\int_X\exp(\lambda f(x))\,\exp((1-\lambda)g(x))\ \mu(dx)\right)\\
    &=\log\left(\int_X\exp(f(x))^\lambda\,\exp(g(x))^{(1-\lambda)}\ \mu(dx)\right)\\
    &\leq\log\left(\left[\int_X\exp(f(x))\ \mu(dx)\right]^\lambda\left[\int_X\exp(g(x))\ \mu(dx)\right]^{1-\lambda}\right)\\
    &=\lambda\log\left(\int_X\exp(f(x))\ \mu(dx)\right)+(1-\lambda)\log\left(\int_X\exp(g(x))\ \mu(dx)\right)\\
    &=\lambda F(f)+(1-\lambda)F(g).
\end{align*}
Therefore, the functional is convex. \\
Continuity follows via dominated convergence since $(\exp(f_n))_n$ is bounded for $f_n\to f$ in $L^\infty(X)$.\\
We continue with the Fenchel conjugate. By Jensen's inequality we have
\begin{align}\label{eq:fenchelupper}
    F^*(\gamma)&=\sup_{g\in L^\infty(X)}\left\{ \int_X g(x)\ \gamma(dx)-F(g)\right\}\nonumber\\
    &=\sup_{g\in L^\infty(X)}\left\{ \int_X g(x)\ \gamma(dx)-\log\left(\int_X\exp(g(x))\ \mu(dx)\right)\right\}\nonumber\\
    &=\sup_{g\in L^\infty(X)}\left\{ \int_X \log(\exp(g(x)))\ \gamma(dx)-\log\left(\int_X\exp(g(x))\frac{d\mu}{d\gamma}(x)\ \gamma(dx)\right)\right\}\nonumber\\
    &\leq\sup_{g\in L^\infty(X)}\left\{ \int_X \log(\exp(g(x)))-\log\left(\exp(g(x))\frac{d\mu}{d\gamma}(x)\right)\ \gamma(dx)\right\}\nonumber\\
    &=\sup_{g\in L^\infty(X)}\left\{ \int_X \log\left(\frac{d\gamma}{d\mu}(x)\right)\ \gamma(dx)\right\}\nonumber\\
    &=\int_X \log\left(\frac{d\gamma}{d\mu}(x)\right)\ \gamma(dx).
\end{align}
Plugging in the  specific sequence of functions $g_n=\log\left(\frac{d\gamma}{d\mu}\right)\wedge n$ and applying dominated convergence and Fatou's lemma yields
\begin{align}\label{eq:fenchellower}
    &F^*(\gamma)\nonumber\\&\geq \liminf_n\left(\int_X\left(\log\left(\frac{d\gamma}{d\mu}(x)\right)\wedge n\right)\ \gamma(dx)-F\left(\log\left(\frac{d\gamma}{d\mu}\right)\wedge n\right)\right)\nonumber\\    &=\liminf_n\left(\int_X\left(\log\left(\frac{d\gamma}{d\mu}(x)\right)\wedge n\right)\ \gamma(dx)-\log\left(\int_X \exp\left( \log\left(\frac{d\gamma}{d\mu}\right)\wedge n\right)\ \mu(dx)\right)\right)\nonumber\\    &\geq\liminf_n\left(\int_X\left(\log\left(\frac{d\gamma}{d\mu}(x)\right)\wedge n\right)\ \gamma(dx)\right)\nonumber\\&-\limsup_n\left(\log\left(\int_X \exp\left( \log\left(\frac{d\gamma}{d\mu}\right)\wedge n\right)\ \mu(dx)\right)\right)\nonumber\\
    &\geq\int_X\log\left(\frac{d\gamma}{d\mu}(x)\right)\ \gamma(dx)-\log\left(\int_X\frac{d\gamma}{d\mu}(x)\ \mu(dx)\right)\nonumber\\
    &=\int_X\log\left(\frac{d\gamma}{d\mu}(x)\right)\ \gamma(dx)-\log(1)\nonumber\\
    &=\int_X\log\left(\frac{d\gamma}{d\mu}(x)\right)\ \gamma(dx).
\end{align}

Therefore, combining \eqref{eq:fenchellower} and \eqref{eq:fenchelupper}, we obtain 
\begin{eqnarray*}
F^\ast(\gamma)=\int_X\log\left(\frac{d\gamma}{d\mu}(x)\right)\ \gamma(dx).
\end{eqnarray*}
\end{proof}

\section{Estimates for $Z_T(u)$}\label{sec:ZTu}

\begin{lem}\label{lem:H^1boundZ}
  Let $\alpha\in \R $. For any $v\in L^2([0,\infty),\dot{H}^\alpha)$ we have
  \begin{align*}
      \sup_{0\leq t\leq T} \Vert \nabla Z_t(v)\Vert_{H^{\alpha+1}}^2\leq \int_0^T\Vert v_s \Vert_{H^\alpha}^2\ ds.
  \end{align*}
\end{lem}
\begin{proof}We follow \cite[Lemma 2]{Barashkov_2020}. First, note that by Mihlin-Hörmanders multiplier theorem (see e.g. \cite[Theorem 2.78]{bahouri2011fourier}) it is sufficient to show
\begin{align*}
    \Vert \int_0^t\sigma_s(D)v_s\ ds\Vert^2_{H^\alpha}\leq \int_0^T\Vert v_s \Vert_{H^\alpha}^2\ ds.
\end{align*}
As $\sigma_t(D)$ is a Fourier multiplier it is diagonal w.r.t the Fourier basis $(e_k)_k$ of trigonometric polynomials. Using Parseval's identity and $|\rho|\leq 1$ this leads to
\begin{align*}
    \Vert \int_0^t\sigma_s(D)v_s\ ds\Vert^2_{H^\alpha}=&\Vert \int_0^t\sum_{k\in\Z^d\setminus\{0\}}\langle v_s,\sigma_s(D)e_k\rangle e_k\ ds\Vert^2_{H^\alpha}\\
    = &\Vert (1+|\cdot|^2)^\frac{\alpha}{2}\mathcal{F}(\int_0^t\sum_{k\in\Z^d\setminus\{0\}}\langle v_s,\sigma_s(D)e_k\rangle e_k\ ds)\Vert_{L^2}^2\\
    =&  \sum_{k\in\Z^d\setminus\{0\}}(1+|k|^2)^\alpha|\int_0^t\langle v_s,\sigma_s(D)e_k\rangle\ ds|^2\\
    \leq&  \sum_{k\in\Z^d\setminus\{0\}}(1+|k|^2)^\alpha\left(\int_0^t|\langle e_k,\sigma_s(D)e_k\rangle|^2\ ds\right)\left(\int_0^t|\langle e_k,v_s\rangle|^2\ ds\right)\\
    \leq& \left(\sup_{k\in\Z^d\setminus\{0\}}\int_0^t\langle e_k,\sigma_s(D)^2e_k\rangle\ ds\right) \sum_{k\in\Z^d\setminus\{0\}}(1+|k|^2)^\alpha\left(\int_0^t|\langle e_k,v_s\rangle|^2\ ds\right)\\
    =& \int_0^t \Vert v_s \Vert_{H^\alpha}^2\ ds\sup_{k\in\Z^d\setminus\{0\}}\int_0^t\langle e_k,\sigma_s(D)^2e_k\rangle\ ds\\
    \leq&\int_0^t \Vert v_s \Vert_{H^\alpha}^2\ ds\sup_{k\in\Z^d\setminus\{0\}}\langle e_k,\rho_t(D)^2e_k\rangle\\
    \leq&\int_0^T\Vert v_s \Vert_{H^\alpha}^2\ ds.\quad 
\end{align*}
\end{proof} 

\section{Wick-Powers}\label{Wick}
Wick-powers of the Gaussian Free Field have been studied also in the vectorial case, e.g. in \cite{10.1214/24-EJP1256}. As expected we can obtain similar regularities in our case. \\
Even though, we focused on two space dimensions we state here a more general result to reveal the difficulties that arise in higher dimensions. 

\begin{lem}\label{WickRegularity}
    Let $\alpha\in\{1,..,d\}^k$ with $k\in\N$ such that $\frac{k-1}{2k}d<1$. Then we have
\begin{align*}
    \sup_T\mathbb{E}\left[ \Vert \int_0^T...\int_0^{s_{k-1}} d(\partial_{\alpha_1}Y_{s_1})...d(\partial_{\alpha_k}Y_{s_k})\Vert^p_{B_{p,p}^{(1-\frac{d}{2})k-\kappa}}\right]<\infty
\end{align*}
for large $p$ and small $\kappa$. \\
In particular,
\begin{align*}
     \sup_T\mathbb{E}\left[ \Vert [\![\nabla W_T^k]\!]\Vert_{\mathscr{C}^{(1-\frac{d}{2})k-\kappa}}\right]<\infty.
\end{align*}

\end{lem}
\begin{proof}
Since we have by Itô's isometry and the independencies of the Brownian motions $(B^n)_{n\in\Z^d\setminus\{0\}}$
\begin{align*}
    &\mathbb{E}\left[ \Vert \int_0^T...\int_0^{s_{k-1}} d(\partial_{\alpha_1}Y_{s_1})...d(\partial_{\alpha_k}Y_{s_k})\Vert^p_{B_{p,p}^\beta}\right]\\
    =&\sum_j2^{j\beta p}\mathbb{E}\left[ \Vert \Delta_j(\int_0^T...\int_0^{s_{k-1}} d(\partial_{\alpha_1}Y_{s_1})...d(\partial_{\alpha_k}Y_{s_k}))\Vert^p_{L^p} \right]\\
    \lesssim& \sum_j2^{j\beta p}\int_{\mathbb{T}^d}\mathbb{E}\left[ |\Delta_j(\int_0^T...\int_0^{s_{k-1}} d(\partial_{\alpha_1}Y_{s_1})...d(\partial_{\alpha_k}Y_{s_k}))(x)|^2 \right]^\frac{p}{2}\ dx\\
    \lesssim& \sum_j2^{j\beta p}\\\int_{\mathbb{T}^d}&\mathbb{E}\left[ \sum_{n_1,...n_k\in\Z^d\setminus\{0\}} \rho_j^2(n_1+...+n_k) \int_0^T...\int_0^{s_{k-1}}\frac{(n_1)_{\alpha_1}^2\sigma^2_{s_1}(n_1)}{(n_1)^2+(n_1)^4}...\frac{(n_k)_{\alpha_k}^2\sigma^2_{s_k}(n_k)}{(n_k)^2+(n_k)^4}\ ds_1...ds_k\right]^\frac{p}{2}\ dx\\
    \leq& \sum_j2^{j\beta p} \left(\sum_{n_1,...n_k\in\Z^d\setminus\{0\}} \rho_j(n_1+...+n_k) \int_0^T...\int_0^{s_{k-1}}\frac{(n_1)_{\alpha_1}^2\sigma^2_{s_1}(n_1)}{(n_1)^2+(n_1)^4}...\frac{(n_k)_{\alpha_k}^2\sigma^2_{s_k}(n_k)}{(n_k)^2+(n_k)^4}\ ds_1...ds_k\right)^\frac{p}{2}
 \end{align*}
it remains to estimate the second sum.
\begin{align*}
    \sum_{n_1,...n_k\in\Z^d\setminus\{0\}} \rho_j(n_1+...+n_k) \int_0^T...\int_0^{s_{k-1}}\frac{(n_1)_{\alpha_1}^2\sigma^2_{s_1}(n_1)}{(n_1)^2+(n_1)^4}...\frac{(n_k)_{\alpha_k}^2\sigma^2_{s_k}(n_k)}{(n_k)^2+(n_k)^4}\ ds_1...ds_k\\
    \leq \sum_{n_1,...n_k\in\Z^d\setminus\{0\}} \rho_j(n_1+...+n_k)\frac{1}{1+(n_1)^2}...\frac{1}{1+(n_k)^2}\\
    =\sum_{n_1\in\Z^d\setminus\{0\}} \rho_j(n_1)\sum_{n_2,...n_k\in\Z^d\setminus\{0\}} \frac{1}{1+(n_1-n_2)^2}...\frac{1}{1+(n_{k-1}-n_k)^2}\frac{1}{1+(n_k)^2}\\
    =\sum_{n_1\in\Z^d\setminus\{0\}} \rho_j(n_1)\sum_{n_2,...n_k\in\Z^d\setminus\{0\}} \frac{1}{1+(n_k)^2}\prod_{j=2}^k \frac{1}{1+(n_{j-1}-n_{j})^2}
\end{align*}
To understand this sum we prove by induction that for all $k\in\N$ we have
\[\sum_{n_2,...n_k\in\Z^d\setminus\{0\}} \frac{1}{1+(n_k)^2}\prod_{j=2}^k \frac{1}{1+(n_{j-1}-n_{j})^2}\] lies in $\mathcal{O}(|n_1|^{-2+\varepsilon+(k-1)(d-2)})$ for any small $\varepsilon>0$. We call this proposition $P(k)$.\\
For $k=1$ the assertion is trivially true. Now, we show the implication $P(k)\Rightarrow P(k+1)$. \\
So we assume that $P(k)$ holds for a fixed $k\in\N$.
\begin{align*}
    \sum_{n_2,...n_{k+1}\in\Z^d\setminus\{0\}} \frac{1}{1+(n_{k+1})^2}\prod_{j=2}^{k+1} \frac{1}{1+(n_{j-1}-n_{j})^2}\\
    =\sum_{n_2\in\Z^d\setminus\{0\}} \frac{1}{1+(n_1-n_2)^2}\sum_{n_3,...n_{k+1}\in\Z^d\setminus\{0\}} \frac{1}{1+(n_{k+1})^2}\prod_{j=3}^{k+1} \frac{1}{1+(n_{j-1}-n_{j})^2}\\
    \stackrel{P(k)}{\lesssim}\sum_{n_2\in\Z^d\setminus\{0\}} \frac{1}{1+(n_1-n_2)^2} (|n_2|)^{-2+\varepsilon+(k-1)(d-2)}\\
    \leq \sum_{|n_2|\leq\frac{|n_1|}{2}} \frac{1}{1+(n_1-n_2)^2} (|n_2|)^{-2+\varepsilon+(k-1)(d-2)}\\
    +\sum_{|n_1-n_2|\leq\frac{|n_1|}{2}} \frac{1}{1+(n_1-n_2)^2} (|n_2|)^{-2+\varepsilon+(k-1)(d-2)}\\
    +\sum_{else} \frac{1}{1+(n_1-n_2)^2} (|n_2|)^{-2+\varepsilon+(k-1)(d-2)}
\end{align*}
We estimate the three terms separately starting with the first. Here the difference $|n_1-n_2|$ scales as $|n_1|$.
\begin{align*}
    \sum_{|n_2|\leq\frac{|n_1|}{2}} \frac{1}{1+(n_1-n_2)^2} (|n_2|)^{-2+\varepsilon+(k-1)(d-2)}\\
    \lesssim \sum_{|n_2|\leq\frac{|n_1|}{2}} \frac{1}{1+(n_1)^2} (|n_2|)^{-2+\varepsilon+(k-1)(d-2)}\\
    =\frac{1}{1+(n_1)^2}\sum_{|n_2|\leq\frac{|n_1|}{2}}  (|n_2|)^{-2+\varepsilon+(k-1)(d-2)}\\
    \lesssim\frac{1}{1+(n_1)^2}\int_0^{\frac{|n_1|}{2}}r^{-2+\varepsilon+(k-1)(d-2)}r^{d-1}\ dr\\
    \lesssim |n_1|^{-2-2+\varepsilon+(k-1)(d-2)+d+\varepsilon}\\
    =|n_1|^{-2+2\varepsilon+k(d-2)}
\end{align*}
In the second term $|n_2|$ scales as $|n_1|$.
\begin{align*}
    \sum_{|n_1-n_2|\leq\frac{|n_1|}{2}} \frac{1}{1+(n_1-n_2)^2} (|n_2|)^{-2+\varepsilon+(k-1)(d-2)}\\\lesssim\sum_{|n_1-n_2|\leq\frac{|n_1|}{2}} \frac{1}{1+(n_1-n_2)^2} (|n_1|)^{-2+\varepsilon+(k-1)(d-2)}\\
    =(|n_1|)^{-2+\varepsilon+(k-1)(d-2)}\sum_{|n_1-n_2|\leq\frac{|n_1|}{2}} \frac{1}{1+(n_1-n_2)^2} \\
    \lesssim(|n_1|)^{-2+\varepsilon+(k-1)(d-2)}\int_0^\frac{|n_1|}{2}\frac{1}{1+r^2}r^{d-1}\ dr\\
    \lesssim (|n_1|)^{-2+\varepsilon+(k-1)(d-2)-2+d+\varepsilon}\\
    =|n_1|^{-2+2\varepsilon+k(d-2)}
\end{align*}
Lastly, in the third term the difference $|n_1-n_2|$ scales as $|n_2|$.
\begin{align*}
    \sum_{else} \frac{1}{1+(n_1-n_2)^2} (|n_2|)^{-2+\varepsilon+(k-1)(d-2)}\\
    \lesssim\sum_{else} \frac{1}{1+|n_2|^2} (|n_2|)^{-2+\varepsilon+(k-1)(d-2)}\\
    \sum_{else} (|n_2|)^{-2+\varepsilon+(k-1)(d-2)-2}\\
    \lesssim\int_{\frac{|n_1|}{2}} r^{-2+\varepsilon+(k-1)(d-2)-2}r^{d-1} \ dr\\
    \lesssim |n_1|^{-2+\varepsilon+(k-1)(d-2)-2+d+\varepsilon}\\
    = |n_1|^{-2+2\varepsilon+k(d-2)}
\end{align*}
Since $\varepsilon>0$ was arbitrary this shows that $P(k+1)$ holds true.\\
Taking the support of $\rho_j$ into account we can write
\begin{align*}
    \sum_{n_1\in\Z^d\setminus\{0\}} \rho_j(n_1)\sum_{n_2,...n_k\in\Z^d\setminus\{0\}} \frac{1}{1+(n_k)^2}\prod_{j=2}^k \frac{1}{1+(n_{j-1}-n_{j})^2}\lesssim 2^{dj}\cdot (2^j)^{-2+\varepsilon+(k-1)(d-2)}\\
    =2^{j(d-2+(k-1)(d-2)+\varepsilon}.
\end{align*}
Therefore, we have
\begin{align*}
    \mathbb{E}\left[ \Vert \int_0^T...\int_0^{s_{k-1}} d(\partial_{\alpha_1}Y_{s_1})...d(\partial_{\alpha_k}Y_{s_k})\Vert^p_{B_{p,p}^\beta}\right]\\
    \lesssim \sum_j 2^{j\beta p}(2^{j(d-2+(k-1)(d-2)+\varepsilon})^\frac{p}{2}.
\end{align*}
The series converges for $\beta<(1-\frac{d}{2})k-\frac{\varepsilon}{2}$.
\end{proof}

\end{document}